\documentclass[11pt,eqno]{article}
\usepackage{bbm}
\usepackage{amsmath}
\usepackage{amsfonts}
\usepackage{graphicx}
\usepackage{amsmath}
\usepackage[pagewise]{lineno}
 \usepackage{paralist}
\usepackage{amssymb}
\usepackage{latexsym}
\usepackage{amsmath, amsfonts,amssymb, amsthm, euscript,makeidx,color,mathrsfs}
\usepackage{enumerate}
\usepackage{subfigure}
\usepackage{ulem}

\usepackage{autobreak}
\allowdisplaybreaks[4]

\usepackage{xcolor}
\usepackage[colorlinks,linkcolor=blue,anchorcolor=green,citecolor=red]{hyperref}

\DeclareMathOperator*{\supp}{\mathrm{supp}}
\DeclareMathOperator*{\dist}{\mathrm{dist}}

\def\ol{\overline}
\def\mb{\mathbb}

\def\Om{\Omega}
\def\ra{\rightarrow}
\def\df{{\rm d}}
\def\pt{\partial}
\def\mcH{\mathcal{H}}

\newtheorem{lemma}{Lemma}[section]
\newtheorem{theorem}[lemma]{Theorem}
\newtheorem{remark}[lemma]{Remark}

\newtheorem{proposition}[lemma]{Proposition}

\theoremstyle{definition}
\newtheorem{definition}[lemma]{Definition}

\makeatletter

\@addtoreset{equation}{section}
\makeatother

\begin{document}
\title{The fundamental gap of a kind of two dimensional sub-elliptic operator  }

\date{}

\author{
	{\normalsize Hongli Sun$^{1}$, Donghui Yang$^2$, Xu Zhang$^{2,*}$ }\\
	\vspace{0.2cm}
	\small\it $^1$ School of Mathematics Physics and Big data, Chongqing University of Science and Technology, \\
	\small\it  Chongging 401331, China\\
	\small\it $^2$ School of Mathematics and Statistics, Central South University, Changsha 410083, China \\}

\footnotetext[2]{*Corresponding author: xuzhang174@163.com }

\maketitle
\begin{center}
\begin{abstract}
 
 This paper is concerned at the minimization fundamental gap problem for a class of two-dimensional degenerate sub-elliptic operators. We establish existence results for weak solutions, Sobolev embedding theorem and spectral theory of sub-elliptic operators. We provide the existence and characterization theorems for extremizing potentials $V(x)$ when $V(x)$ is subject to $L^\infty$ norm constraint.

 \vspace{0.2cm}

 \vspace{0.2cm}  


\end{abstract}

\end{center}

\section{Introduction}

The eigenvalue extremum problem of Schr{\"o}dinger operator with Dirichlet/Neumann boundary was initiated at least in the early 1960s \cite{keller1961lower}. The difference between the first two eigenvalues, $\lambda_2-\lambda_1$, is called the fundamental gap, it is of great significance in quantum mechanics, statistical mechanics, and quantum field theory. We give a brief overview of this subject and outline some related work.

Van  den Berg \cite{van1983condensation} observed that $\lambda_2-\lambda_1\ge \frac{3\pi^2}{d^2}$ on many convex domains,  where $d$ is the  diameter of domain, this view has also been independently suggested by Yau \cite{yau1986nonlinear}, as well as Ashbaugh and Benguria \cite{ashbaugh1989optimal}, which is called fundamental gap  conjecture. Ashbaugh and Benguria  \cite{ashbaugh1989optimal} proved that the conjecture holds if the potential function is single-well symmetric (not necessarily convex).  Later, Horv\'ath \cite{horvath2003first} removed the symmetry hypothesis and allowed  the potential function to be single-well with the middle transition point on the interval.  Lavine \cite{lavine1994eigenvalue}  proved the fundamental gap conjecture for Schr{\"o}dinger operators equipped with homogeneous Dirichlet or Neumann boundary conditions on a bounded interval among the class of convex bounded potentials, and this result  was completed for Robin boundary conditions in \cite{andrews2021fundamental}. More results could be found in \cite{Henrot2006Extremum}.

 Yau \cite{singer1985estimate} et al. obtained that the gap is bounded below by $\frac{\pi^2}{4d^2}$ in the  high-dimensional case by utilizing the fact that the first eigenfunction is logarithmic concave. Subsequently, many studies improved the the lower bound of the gap, it was not until 2011 that Andrews and Clutterbuck \cite{Ben2011Proof}  completely solved this conjecture. The same method has been further exploited in
 the paper \cite{andrews2020non}.     In particular,  Laplacian fundamental gap problem  \cite{seto2019sharp,Bourni2021The,he2020fundamental,dai2021fundamental} has been extended to particular manifolds and the corresponding detailed characterization is given. The lower bound of eigenvalue gap of vibrating string is also discussed in \cite{2007The,chen2014lower,qi2020extremal}. For the latest references, we can refer to \cite{kerner2021lower,kerner2021lowerAlower,el2022optimal}.

For more general elliptic operators, Wolfson \cite{Wolfson2015eigenvalue} estimated the eigenvalue gap for a class of nonsymmetric second-order linear elliptic operators. Cheng et.al \cite{cheng2014dual} considered the eigenvalue gap of the $p$-Laplacian eigenvalue problems, and obtained the minimizer of the eigenvalue gap for the single-well potential function. Significantly, Tan and Liu \cite{Tan2022Estimates} considered estimates for eigenvalues of a class of fourth order degenerate elliptic operators with a singular potential, Chen et al. \cite{ChenHua2019lower} provided the lower bounds of Dirichlet eigenvalues for a class of higher order degenerate elliptic operators. However, they mainly provided estimates for each eigenvalue, and need to assume the elliptic operator as a sum of squares of vector fields $\{X_j\}_{j=1}^m$, and the vector fields $\{X_j\}_{j=1}^m$ demand the  H{\"o}rmander's condition and M\'etivier's condition.

In general, as far as we know, there are relatively few studies on more general elliptic operators involving the fundamental gap. Inspired by the above work and \cite{Ashbaugh1991ON,Henrot2006Extremum}, we consider the following  degenerate sub-elliptic equation
\begin{equation}\label{ssss1}
	\begin{cases}
		-\left(\partial_{x_1x_1}u+h(x_1)\partial_{x_2x_2}u\right)+Vu=\lambda u, & x\in \Omega, \\
		u=0, & x\in\partial \Omega,
	\end{cases} 
\end{equation}
 where 
 \begin{equation*}
 	\Omega=\left\{x=(x_1,x_2)\in \mathbb{R}^2\mid 0<x_1<1, 0<x_2<1\right\},
 \end{equation*}
 $h:[0,1]\ra \mb{R}$ is a bounded function with $h(0)=0$, $h(x_1)>0$ in $(0,1]$ and $h(x_1)\in  C^1(0,1]$,  moreover there exists a constant $C$ such that
\begin{equation}\label{e2}
\sup_{\substack{x \in \Omega \\ 0<r \le d }}    \frac{\int_{B \cap \Omega}h(x_1)^{-1}dx}{|B\cap \Omega|^\theta}   \le C  \ \mbox{for some}  \  0<\theta<1,
\end{equation}
where $B:=B(x,r)$ represents the ball with point $x \in \Omega$ as the center of the ball and radius r,  $d$ is the diameter of $\Omega$. For example, $h(x_1)=x_1^\alpha$ with $\alpha \in (0,1)$.

Denote 
\begin{equation}\label{e3}
	L=-\left(\partial_{x_1x_1}+h(x_1)\partial_{x_2x_2}\right)+V, 
\end{equation}
 the focus of this paper is to search for the existence and optimality of the minimization problem
  \begin{equation*}
 	\inf\limits_{V\in S }\Gamma(V)
 \end{equation*} 
related to operator \eqref{e3} under the condition $V\in S$, where $\Gamma(V)=\lambda_2(V)-\lambda_1(V)$, and 
\begin{equation*}
	S=\{V\in L^{\infty}(\Omega)\mid  m\le V(x) \le M \mbox{ a.e.}\},
\end{equation*}
$m$ and $M$ are non-negative given constants.

It is not hard to see that we do not assume the operator $L$ is written as a sum of squares of vector fields although it allows for this structure. Furthermore, we do not apply the additional metric structures, the background Euclidean metric is employed only in this paper.

\section{Weighted Sobolev space}

In this section, we provide the solvability theory for the degenerate elliptic equation \eqref{e1} on which the spectral theory and subsequent applications are based on it. The purpose of this theory is to establish the weak well-posedness for the problem \eqref{e1}:
\begin{equation}\label{e1}
	\begin{cases}
		-\left(\partial_{x_1x_1}u+h(x_1)\partial_{x_2x_2}u\right)+Vu=f, & x\in \Omega, \\
		u=0, & x\in\partial \Omega,
	\end{cases} 
\end{equation}
we explore ways of showing the existence, uniqueness of a proper notion of weak solution $u$ for each given $f \in L^2(\Omega)$. The space we mentioned below may be involved in many works such as \cite{2009Maximum,sawyer2010degenerate,cavalheiro2008weighted,2011ACompact}, but we will introduce more profound results.

Define
\begin{equation*}
	\mcH_h^1(\Omega)=\left\{u\in L^2(\Omega)\Bigm| \partial_{x_1}u\in L^2(\Omega), \sqrt{h(x_1)}\partial_{x_2}u\in L^2(\Omega)\right\}, 
\end{equation*}
and 
\begin{equation*}
	(u,v)_{\mcH_h^1(\Omega)}=\int_\Omega uvdx +\int_\Omega(\pt_{x_1}u)(\pt_{x_1}v)dx+\int_\Omega\left(\sqrt{h(x_1)}  \pt_{x_2}  u\right)\left(\sqrt{h(x_1)} \pt_{x_2}v\right)dx,
\end{equation*}
from which we define
\begin{equation*}
	\|u\|_{\mcH_h^1(\Omega)}=\left(\|u\|_{L^2(\Omega)}^2+\|\pt_{x_1}u\|_{L^2(\Omega)}^2+\|\sqrt{h(x_1)}\pt_{x_2}u\|_{L^2(\Omega)}^2\right)^\frac{1}{2}. 
\end{equation*}
It is not hard to find that $\|u\|_{\mcH_h^1(\Omega)}=\sqrt{(\cdot,\cdot)_{\mcH_h^1(\Omega)}}$. 
Furthermore, let $\mcH_{h,0}^1(\Omega)$ denote the closure of $C_0^\infty(\Omega)$ in the space $\mcH_h^1(\Omega)$, that is, 
\begin{equation*}
	\mcH_{h,0}^1(\Omega)=\overline{C_0^\infty(\Omega)}^{\mcH_h^1(\Omega)}. 
\end{equation*}

Throughout this paper,  let  $\|\cdot\|$ denotes the norm,  $(\cdot,\cdot)$ denotes the inner product.

\begin{lemma}\label{L.2.1}
	The space $(\mcH_h^1(\Omega),(\cdot,\cdot)_{\mcH_h^1(\Omega)})$ is a Hilbert space.   
\end{lemma}

\begin{proof}	
	Firstly, we easily verify that $(\mcH_h^1(\Omega),(\cdot,\cdot)_{\mcH_h^1(\Omega)})$ is an inner space. Let $\{u_n\}_{n\in\mb{N}}\subset\mcH_h^1(\Omega)$ be a Cauchy sequence, so that
$\{u_n\}_{n\in\mb{N}},\{\pt_{x_1}u_n\}_{n\in\mb{N}},\{\sqrt{h(x_1)}\pt_{x_2}u_n\}_{n\in\mb{N}}$  are Cauchy sequences in $L^2(\Omega)$.  Then there exist $u,v,w\in L^2(\Omega)$ such that 
\begin{equation*}
	u_n\ra u, \pt_{x_1}u_n\ra v, \sqrt{h(x_1)}\pt_{x_2}u_n\ra w \mbox{ strongly in } L^2(\Omega). 
\end{equation*}
For each $\varphi\in C_0^\infty(\Omega)$, we have 
\begin{equation*}
	\begin{split}
		\int_\Omega w\varphi dx \leftarrow \int_\Omega \left(\sqrt{h(x_1)}\pt_{x_2}u_n\right)\varphi dx =-\int_\Omega u_n\left(\sqrt{h(x_1)}\pt_{x_2}\varphi\right)dx \ra -\int_\Omega  u\left(\sqrt{h(x_1)}\pt_{x_2}\varphi\right)dx 
	\end{split}
\end{equation*}
in the sense of distribution, which implies that 
\begin{equation*}
	w=\sqrt{h(x_1)}\pt_{x_2}u
\end{equation*}
in the sense of distribution. Naturally, $w=\sqrt{h(x_1)}\pt_{x_2}u$ in $L^2(\Omega)$ since $w\in L^2(\Omega)$. Implement the same method again then $v=\pt_{x_1}u$ in $L^2(\Omega)$ is attained. All these imply that 
\begin{equation*}
	u_n\ra u \mbox{ strongly in }\mcH_h^1(\Omega).\qedhere
\end{equation*}
\end{proof}

\begin{proposition}\label{P.2.1}
	The Sobolev space $H^1(\Omega)$ is a subspace of $\mcH_h^1(\Omega)$.  Generally, $H^1(\Omega)\subsetneq \mcH_h^1(\Omega)$. 
\end{proposition}

\begin{proof}
	
	1. Let $u\in H^1(\Omega)$. Then $u, \pt_{x_1}u, \pt_{x_2}u\in L^2(\Omega)$. Note that
	\begin{equation}\label{21.07.23.2}
		\begin{split}
		\|u\|_{\mcH_h^1(\Omega)}
		&=\left(\|u\|_{L^2(\Omega)}^2+\|\pt_{x_1}u\|_{L^2(\Omega)}^2+\|\sqrt{h(x_1)}\pt_{x_2}u\|_{L^2(\Omega)}^2\right)^\frac{1}{2}\\
		&\leq C\left(\|u\|_{L^2(\Omega)}^2+\|\pt_{x_1}u\|_{L^2(\Omega)}^2+\|\pt_{x_2}u\|_{L^2(\Omega)}^2\right)^\frac{1}{2}=C\|u\|_{H^1(\Omega)}, 
		\end{split}
	\end{equation}
hence $u\in \mcH_h^1(\Omega)$. 

2. Take $\sqrt{h(x_1)}=x_1^\gamma, \gamma\in (0,\frac{1}{2})$, consider the function  $u({x_1},{x_2})=\left(x_1^2+{x_2}\right)^\frac{1}{4}$ on $\Omega$. We observe that 
\begin{equation*}
	\begin{split}
		\partial_{x_1}u=\frac{x_1}{2}\left(x_1^2+{x_2}\right)^{-\frac{3}{4}},\quad \partial_{x_2}u=\frac{1}{4}\left(x_1^2+{x_2}\right)^{-\frac{3}{4}},
	\end{split}
\end{equation*} 
and 
\begin{equation*}
	\begin{split}
		\|u\|_{L^2}^2
		&=\int_\Omega (x_1^2+{x_2})^\frac{1}{2} dx\leq \sqrt{2},\\
		\|\partial_{x_1}u\|_{L^2}^2
		&=\frac{1}{4}\int_\Omega \frac{x_1^2}{\left(x_1^2+{x_2}\right)^\frac{3}{2}}dx\leq \frac{1}{4}\int_\Omega \frac{1}{\sqrt{x_1^2+{x_2}}}dx =\frac{1}{4}\int_0^1\int_{x_1^2}^{x_1^2+1}\frac{1}{\sqrt{z}}dzdx_1\\
		&=\frac{1}{4}\int_0^1 2\sqrt{z}\Big|_{x_1^2}^{x_1^2+1} dx_1\leq \sqrt{2},\\
		\|\sqrt{h(x_1)}\partial_{x_2}u\|_{L^2}^2
		&=\frac{1}{16}\int_\Omega \frac{x_1^{2\gamma}}{\left(x_1^2+{x_2}\right)^\frac{3}{2}}dx\leq \frac{1}{16}\int_\Omega \frac{(x_1^2+{x_2})^\gamma}{\left(x_1^2+{x_2}\right)^\frac{3}{2}}dx=\frac{1}{16}\int_\Omega\frac{1}{(x_1^2+{x_2})^{\frac{3}{2}-\gamma}}dx\\
		&=\frac{1}{16}\int_0^1\int_{x_1^2}^{x_1^2+1}\frac{1}{z^{\frac{3}{2}-\gamma}}dzdx_1=\frac{1}{16}\left(-\frac{1}{2}+\gamma\right)^{-1}\int_0^1z^{-\frac{1}{2}+\gamma}\Big|_{x_1^2}^{x_1^2+1}dx_1<+\infty, 
	\end{split}
\end{equation*}
these show that $u\in \mcH_h^1(\Omega)$. However,
\begin{equation*}
	\begin{split}
		\|\partial_{x_2}u\|_{L^2}^2
		&=\frac{1}{16}\int_\Om \frac{1}{\left(x_1^2+{x_2}\right)^\frac{3}{2}} dx=\frac{1}{16}\int_0^1 dx_1\int_0^1\frac{1}{\left(x_1^2+{x_2}\right)^\frac{3}{2}}d {x_2}=\frac{1}{16}\int_0^1 dx_1\int_{x_1^2}^{x_1^2+1}\frac{1}{z^\frac{3}{2}}dz\\
		&=-\frac{1}{8}\int_0^1z^{-\frac{1}{2}}\Big|_{x_1^2}^{x_1^2+1}dx_1=-\frac{1}{8}\int_0^1\left(\frac{1}{\sqrt{x_1^2+1}}-\frac{1}{x_1}\right)dx_1=+\infty,
	\end{split}
\end{equation*}
this implies that $u\notin H^1(\Omega)$. 
\end{proof}

\begin{lemma}\label{202306061841}
$\mcH_h^1(\Omega)$ is continuously embedded into $W^{1,1}(\Omega)$.
\end{lemma}

\begin{proof}
We observe that $\int_{\Omega} h(x_1)^{-1}dx$ is bounded from \eqref{e2}. For any $u \in  \mcH_h^1(\Omega)$, thanks to H{\"o}lder inequality,
	\begin{equation*}
	\begin{split}		
		\int_{\Omega} |\partial_{x_2} u|dx&=\int_{\Omega}\left(\sqrt{h(x_1)}\right)^{-1} \sqrt{h(x_1)} |\partial_{x_2} u|dx\\
		&\le \left(\int_{\Omega} h(x_1)^{-1} dx\right)^{\frac{1}{2}} \left(\int_{\Omega} |h(x_1)\partial_{x_2}u|^2 dx\right)^{\frac{1}{2}}\\ 
		&\le C \|h(x_1)\partial_x u\|_{L^2(\Omega)}.
	\end{split}
\end{equation*}
Therefore, $\mcH_h^1(\Omega) \hookrightarrow W^{1,1}(\Omega)$.
\end{proof}

\begin{lemma}\label{202306061911}
$H_0^1(\Omega) \subset \mcH_{h,0}^1(\Omega) \subset \mcH_h^1(\Omega) \cap W_0^{1,1}(\Omega)$.
\end{lemma}

\begin{proof}
For $u \in H_0^1(\Omega)$, there is a sequence $\{u_n\}_{n\in\mb{N}} \subset C_0^\infty(\Omega)$ such that $\|u_n-u\|_{H^1(\Omega)} \to 0$, $n \to \infty$, and we observe that 
\begin{equation*}
	\|u_n-u\|_{\mcH_h^1(\Omega)} \le C\|u_n-u\|_{H^1(\Omega)} \to 0, \ n \to \infty,
\end{equation*}	
therefore $u \in \mcH_{h,0}^1(\Omega)$.

For $u \in \mcH_{h,0}^1(\Omega)$, there is a sequence $\{u_n\}_{n\in\mb{N}} \subset C_0^\infty(\Omega)$ such that $\|u_n-u\|_{\mcH_h^1(\Omega)} \to 0$, $n \to \infty$. And from Lemma \ref{202306061841}
\begin{equation*}
	\|u_n-u\|_{W^{1,1}(\Omega)} \le C\|u_n-u\|_{\mcH_h^1(\Omega)} \to 0, \ n \to \infty,
\end{equation*}	
so that we have $ \mcH_{h,0}^1(\Omega) \subset \mcH_h^1(\Omega) \cap W_0^{1,1}(\Omega)$.		
\end{proof}

From Lemma \ref{202306061911}, we know that it makes sense to consider Dirichlet condition problem \eqref{e1}.

\begin{lemma}\label{L.2.4}
	For any $u\in \mcH_0^1(\Omega)$, 
\begin{equation*}
	\tilde u(x)=
	\begin{cases}
		u(x), & x\in \Omega,\\
		0, & x\in \mb{R}^2 \backslash \Omega,
	\end{cases}
\end{equation*}
we have $\tilde u \in \mcH_h^1(\mb{R}^2)$.
\end{lemma}

\begin{proof}
In fact, for any $u\in \mcH_0^1(\Omega)$, there is a sequence $\{u_n\}_{n\in\mb{N}}\subset C_0^\infty(\Omega)$ such that  
\begin{equation*}
	u_n\ra u  \ \text{in} \ \mcH_h^1(\Omega), n\to \infty. 
\end{equation*}
For any $\varphi\in C_0^\infty(\mathbb{R}^2)$,
\begin{equation*}
	\begin{split}
		\int_{\mb{R}^2} (\sqrt{h(x_1)}\pt_{x_2}\tilde u)\varphi dx&=-\int_{\mb{R}^2}\tilde u(\sqrt{h(x_1)}\pt_{x_2}\varphi)dx
		=-\int_\Omega u(\sqrt{h(x_1)}\pt_{x_2}\varphi)dx\\
		&=-\lim_{n\ra\infty}\int_\Omega u_n(\sqrt{h(x_1)}\pt_{x_2}\varphi)dx=\lim_{n\ra\infty}\int_\Omega (\sqrt{h(x_1)}\pt_{x_2} u_n)\varphi dx\\
		&=\int_\Omega (\sqrt{h(x_1)}\pt_{x_2}u)\varphi dx=\int_{\mb{R}^2}{\left[\sqrt{h(x_1)}\pt_{x_2}u\right]}^\sim\varphi dx,
	\end{split}
\end{equation*}
we have
\begin{equation*}
	\sqrt{h(x_1)}\pt_{x_2}\tilde u={\left[\sqrt{h(x_1)}\pt_{x_2} u\right]}^\sim, 	   
\end{equation*}
in the sense of distribution, 
\begin{equation*}
	{\left[\sqrt{h(x_1)}\pt_{x_2} u\right]}^\sim(x)=
	\begin{cases}
		\sqrt{h(x_1)}\pt_{x_2} u(x), & x\in \Omega,\\
		0, & x\in \mb{R}^2 \backslash \Omega. 
	\end{cases}
\end{equation*}
Similarly, 
\begin{equation*}
	\pt_{x_1}\tilde u=(\pt_{x_1}u)^\sim . 
\end{equation*}
This indicates that $\tilde u\in \mcH_h^1(\mb{R}^2)$ and $\|\tilde u\|_{\mcH_h^1(\mb{R}^2)}=\|u\|_{\mcH_h^1(\Omega)}$.		
\end{proof}

\section{Compact embedding theorem}

Although the compact embedding problem of degenerate elliptic operators has been discussed a lot, such as \cite{Monticelli2020An,2011ACompact}, it is still difficult to solve in specific problems. Next, we prove that the embedding $\mcH_{h,0}^1(\Omega)\hookrightarrow L^2(\Omega)$ is compact by utilizing different techniques.

We will also utilize the some definitions \cite{mamedov2021poincare,muckenhoupt1972weighted,NonlinearPotential}, let $v(x)>0$ a.e. and local integrable in $\mathbb{R}^2$ with respect to Lebesgue measure, we shall use the notation $v(E)=\int_{E} v(x)dx$ for a measurable set $E \subset\mathbb{R}^2$, and ordinary Lebesgue measure will be denoted by $|E|$. 
We say that the function $v(x)$ belongs to the {\it Muckenhoupt class} $A_\infty(=A_\infty(\mathbb{R}^2,d,dx))$ if there exist constants $C$ and $\sigma  >0$ such that 
\begin{equation}\label{Se3.2}
	\frac{v(E)}{v(B)}\le C \left(\frac{|E|}{|B|}\right)^{\sigma } 
\end{equation}
for any Euclidean balls $B$  and any Lebesgue measurable set (for simply, measurable set) $E\subset B$, where $d$ and $dx$  are the standard Euclidean metric and Lebesgue measure in $\mb{R}^2$ respectively. 

 On the basis of the results involved in  \cite{mamedov2009some,mamedov2021poincare,mamedov2018sawyer,mamedov2014weighted},  and in combination with the contents of this paper,  we develop the following Lemma \ref{Le.3.1}.

\begin{lemma}\label{Le.3.1}
Let  $1< p\leq q <\infty$, $\Omega=(0,1)\times(0,1)$ $\subset \mathbb{R}^2$. Suppose that $v\in A_\infty$,  $w_j(x)>0$ a.e. and $w_j(x) \in L^1(\Omega)$.  For any  ball $B(x,r)$ having a center $x\in\Omega$, $0<r\leq d(\Omega)$,   if there is a constant  $A_{pq}$ such that 
\begin{equation}\label{Le.3.1.e1}
|B|^{-1}{d(B)}v(B\cap\Omega)^\frac{1}{q}{\left[{\omega_j}^{-\frac{1}{p-1}}(B\cap\Omega)\right]}^\frac{p-1}{p}\leq A_{pq},\quad j=1,2,
\end{equation} 
 where $d(\Omega)$ is the diameter of $\Omega$, then there exists a positive number $C_0(q, C, \sigma )$ such that for any $u\in{{\rm Lip}_0(\Omega)}$,
\begin{equation}
	{\left(\int_{\Omega}{\left\vert u(x)\right\vert}^qv(x)dx\right)}^\frac{1}{q}\leq C_0A_{pq}\sum_{j=1}^2\left(\int_{\Omega}{\left\vert \partial_{x_j}u\right\vert}^p\omega_j(x)dx\right)^\frac{1}{p},
\end{equation}
where ${\rm Lip}_0(\Omega)$ is the class of Lipschitz continuous functions which have compact support in $\Omega$, $C$ and $\sigma$ as shown in \eqref{Se3.2}.
\end{lemma}

\begin{proof}
	We carry out this proof by several steps.
	
	{\it Step 1}. 
	Let $u(x)\in {\rm Lip}_0(\Omega)$. We put $\Omega^+=\{x\in \Omega\mid u(x)>0\}$ and $\Omega^-=\{x\in \Omega\mid u(x)<0\}$. Let $\Omega^i$ be a connected component of $\Omega^+\ (i=1,2,\dots)$. For $\alpha>0$, we denote $\Omega_\alpha=\{x\in \Omega^i\mid u(x)>\alpha\}$. Since $u(x)$ is continuous, the set $\Omega_\alpha$ is open. Let $\alpha$ be such that  $|\Omega_{2\alpha}|>0$. Then for any fixed point $x\in \Omega_{2\alpha}$, there exists a ball $B=B(x,r(x))$ such that
	\begin{equation}\label{X4}
		|B(x,r(x)) - \Omega_\alpha|=\gamma|B(x,r(x))|,
	\end{equation}  
	where $\gamma\in(0,1)$ is some number independent of $\alpha,x $ and $r(x)$; this $\gamma$ will be specified later. 
	
	Indeed, to prove it, consider the function
	\begin{equation*}
		F(t)=|B(x,t)-\Omega_\alpha|-\gamma|B(x,t)|,
	\end{equation*}
note that for each $t_1,t_2 \in \mathbb{R}$, we assume $t_1\leq t_2$, then  
	\begin{equation*}	
		\begin{split}	
			|F(t_2)-F(t_1)|
			&=\bigg| |B(x,t_2)- \Omega_\alpha|- |B(x,t_1)- \Omega_\alpha| + \gamma (|B(x,t_1)|-|B(x,t_2)|)  \bigg|\\		
			& \le \bigg| |B(x,t_2) - \Omega_\alpha|- |B(x,t_1) - \Omega_\alpha|\bigg|+ \bigg| \gamma (|B(x,t_1)|-|B(x,t_2)|)  \bigg|\\		
			&\le \bigg| |B(x,t_2)|- |B(x,t_1)|\bigg|+ \bigg| \gamma (|B(x,t_1)|-|B(x,t_2)|)  \bigg|	\\		
			&\le (1+\gamma)\pi|t_2^2-t_1^2|.	
		\end{split}	
	\end{equation*}
Therefore, $F(t)$ is a continuous function with repect to $t$.      On one hand,   $F(t)$ is negative for sufficiently small $t>0$ since $x$ is an interior point of $\Omega_{2\alpha}$. On the other hand, due to the boundedness of $\Omega$, there exists $t\in(1/\sqrt{\pi},+\infty)$  such that
	$$|B(x,t)-\Omega_\alpha|>|B(x,t)-\Omega|\geq \gamma|B(x,t)|,$$ 
which implies that  the function $F(t)$ be positive for sufficiently large values of $t$.  Choose $r(x)\in\mb{R}$ such that $F(r(x))=0$. This proves \eqref{X4}.

	{\it Step 2}. For a fixed point $x\in \Omega_{2\alpha}$, we denote $B=B(x,r(x))$ for simplicity. There are two possibilities:
	
	$(1)$ If
	\begin{equation}\label{X6}
		|\Omega_{2\alpha}\cap B|<\gamma|B|,
	\end{equation}
	using the assumption  $v(x)\in A_\infty$, we obtain
	\begin{equation}\label{X7}
		v(\Omega_{2\alpha}\cap B)\leq C\gamma^\sigma  v(B).
	\end{equation} 
	Using construction \eqref{X4} and assumption $v(x)\in A_\infty$, it follows 
	\begin{equation*}
		v(B)=v(B\cap \Omega_\alpha)+v(B-\Omega_\alpha)\leq v(B\cap \Om_\alpha)+C\left(\frac{|B-\Om_\alpha|}{|B|}\right)^\sigma  v(B)\leq v(B\cap \Omega_\alpha)+C\gamma^\sigma  v(B).	
	\end{equation*}
	Choosing $\gamma$ such that $C\gamma^\sigma <1$, we have 
	\begin{equation*}
		v(B)\leq \frac{1}{1-C\gamma^\sigma }v(B\cap \Omega_\alpha)	
	\end{equation*}
	and from \eqref{X7}, it follows that 
	\begin{equation}\label{X8}
		v(B\cap \Omega_{2\alpha})\leq \frac{C\gamma^\sigma }{1-C\gamma^\sigma }v(B\cap \Omega_\alpha).
	\end{equation}
	
	$(2)$ If
	\begin{equation}\label{X9}
		|\Omega_{2\alpha}\cap B|\ge \gamma|B|,
	\end{equation} 
	then from \eqref{X4} and \eqref{X9}, we have 
	\begin{equation*}
		\int_{A}\left(\int_{Z}dy\right)dx\ge \gamma^2|B|^2,
	\end{equation*}
	where $A=B-\Omega_\alpha$, $Z=B\cap\Omega_{2\alpha}$.
	
	Let points $x\in A$, $y\in Z$ be arbitrarily fixed. The line segment $\{x+t(y-x)\mid 0<t<1\}$ connecting $x$, $y$ lies in $B$ necessarily intersects the surfaces $\{x\in \Omega^i\mid u(x)=\alpha\}$ and $\{x\in \Omega^i\mid u(x)=2\alpha\}$ at some points $x^\prime=x+t_1(y-x), \hat x'=x+t_1'(y-x)$ and $x^{\prime\prime}=x+t_2(y-x)$, where $t_1,t_1'$ be the smallest and largest value of $t$ when the line segment intersects the surface $\partial \Omega_\alpha$, repsectively, and $t_2$ denote the first time when the line segment $\{x+t(y-x)\mid 0<t<1\}$ meets the surface $\partial \Omega_{2\alpha}$, i.e., $0<t_1\leq t_1'<t_2<1$ are numbers depending on $x$ and $y$. Then $u(x^\prime)=u(\hat x')=\alpha$ and $u(x^{\prime\prime})=2\alpha$.
	
	(i) It is clearly that 
	\begin{equation*}
		\gamma^2|B|^2\leq \frac{1}{\alpha}\int_{A}\left(\int_{Z}|u(\hat x^\prime)-u\left(x^{\prime\prime}\right)|dy\right)dx,
	\end{equation*}
	whence
	\begin{equation*}
		\begin{split}
			\gamma^2|B|^2
			&\leq \frac{1}{\alpha}\int_{A}\left(\int_{Z}\left(\int_{t_1'(x,y)}^{t_2(x,y)}\big|u_t(x+t(y-x))\big|dt\right)dy\right)dx.
		\end{split}
	\end{equation*} 
	According to Fubini's theorem,
	\begin{equation*}
		\gamma^2|B|^2\leq \sum_{j=1}^{2}\frac{d(B)}{\alpha}\int_{A}\left(\int_{t_1'(x,y)}^{t_2(x,y)}\left(\int_{\{y\in Z\mid x+t(y-x)\in G\}}
		\big|\partial_{x_j}u(x+t(y-x))\big|dy\right)dt\right)dx,
	\end{equation*}
	where $G=B\cap \left(\Om_\alpha-\Om_{2\alpha}\right)$. 
	
	(ii) For each $a,b\in [t_1,t_1'] \ (a<b)$ with $u(x_t)>\alpha, \forall t\in (a,b)$, where $u(x_a)=u(x_b)=\alpha$ and $x_t=x+t(y-x)$. It is also clearly that 
	\begin{equation*}
		0\leq  \sum_{j=1}^{2}\frac{d(B)}{\alpha}\int_{A}\left(\int_a^b\left(\int_{\{y\in Z\mid x+t(y-x)\in G\}}\big|\partial_{x_j}u(x+t(y-x))\big|dy\right)dt\right)dx. 
	\end{equation*}
	
	According to (i) and (ii), we obtain that 
	\begin{equation*}
		\gamma^2|B|^2\leq \sum_{j=1}^{2}\frac{d(B)}{\alpha}\int_{A}\left(\int_0^1\left(\int_{\{y\in Z\mid x+t(y-x)\in G\}}\big|\partial_{x_j}u(x+t(y-x))\big|dy\right)dt\right)dx. 
	\end{equation*}
	
	Insert a change of variable in the interior integral $z=x+t(y-x)$ passing from $y$ to $z$. Since $z\in G$, and $dy=t^{-2}dz$ we obtain
	\begin{equation*}
		\gamma^2|B|^2\leq \sum_{j=1}^{2}\frac{d(B)}{\alpha}\int_{A}\left(\int_0^1\left(\int_{\{z\in G\mid t^{-1}(z-x)+x\in Z\}}
		\big|\partial_{z_j}u(z)\big|dz\right)\frac{dt}{t^2}\right)dx.
	\end{equation*}
	When $t\in (0,1)$, we have $z\in B$. Then $|x_j-z_j|<td(B)$, applying Fubini's formula once again, we get
	\begin{equation*}
		\begin{split}
			\gamma^2|B|^2&\leq \sum_{j=1}^{2}\frac{d(B)}{\alpha}\int_0^1\left(\int_G\big|\partial_{z_j}u(z)\big|\left(\int_{\{x\mid |x_j-z_j|<td(B)\}}dx\right)dz\right)\frac{dt}{t^2}\\
			&\leq \sum_{j=1}^{2}\frac{2^2d(B)|B|}{\alpha}\int_G\big|\partial_{z_j}u(z)\big|dz,
		\end{split}
	\end{equation*}
	Therefore, by the inequality $(a+b)^q\leq 2^{q-1}(a^q+b^q)$, we get
	\begin{equation}\label{X10}
		1\leq \sum_{j=1}^{2}\left(\frac{2^{\frac{3q-1}{q}}d(B)}{\alpha\gamma^2|B|}\int_G\big|\partial_{z_j}u(z)\big|dz\right)^q.
	\end{equation}
	By the H$\ddot{{\rm o}}$lder inequality, it follows
	\begin{equation}\label{X11}
		\begin{split}
			1\leq \sum_{j=1}^{2}&\left(\frac{2^{\frac{3q-1}{q}}d(B)}{\alpha\gamma^2|B|}\right)^q\left(\int_{B\cap \Omega}\omega_j^{-\frac{1}{p-1}}(z)dz \right)^{\frac{q(p-1)}{p}}\left(\int_G{\big|\partial_{z_j}u\big|}^p\omega_j(z)dz\right)^{\frac{q}{p}}.
		\end{split}
	\end{equation}
	Combining the estimates \eqref{Le.3.1.e1} and \eqref{X11}, we get
	\begin{equation}\label{X12}
		1\leq \frac{2^{3q-1}A_{pq}^q}{\gamma^{2q}\alpha^q}\frac{1}{v(B\cap \Omega)}\sum_{j=1}^{2}\left(\int_G{\big|\partial_{z_j}u\big|}^p\omega_j(z)dz\right)^{\frac{q}{p}}. 
	\end{equation}
	Consequently,
	\begin{equation}\label{X13}
		v(B\cap \Omega_{2\alpha})\leq v(B\cap \Omega)\leq \frac{2^{3q-1}A_{pq}^q}{\gamma^{2q}\alpha^q}\sum_{j=1}^{2}\left(\int_G{\big|\partial_{z_j}u\big|}^p\omega_j(z)dz\right)^{\frac{q}{p}}.   
	\end{equation}

	By the estimates \eqref{X8} and \eqref{X13}, we obtain
	\begin{equation}\label{X14}
		\begin{aligned}
			v(B\cap \Omega_{2\alpha})&\leq \frac{C\gamma^\sigma }{1-C\gamma^\sigma }v(B\cap \Omega_\alpha)+\frac{2^{3q-1}A_{pq}^q}{\gamma^{2q}\alpha^q}\sum_{j=1}^{2}\left(\int_{B\cap (\Omega_\alpha-\Omega_{2\alpha})}{\big|\partial_{z_j}u\big|}^p\omega_j(z)dz\right)^{\frac{q}{p}}.
		\end{aligned}
	\end{equation}

	{\it Setp 3}.	It is evident that the system of balls $\{B=B(x,r(x))\mid x\in \Omega_{2\alpha}\}$ covers $\Omega_{2\alpha}$. Furthermore, by the constructions above, we have $\mathop{{\rm sup}}\limits_{x\in \Omega_{2\alpha}}r(x)<\infty$. Due to the Besicovitch's covering theorem \cite{evans2018measure,duoandikoetxea2001fourier}, one can select a finite or countable subcover $\{B_k\}_{k=1}^\infty$ that covers $\Omega_{2\alpha}$ from $\{B(x,r(x))\mid x\in \Omega_{2\alpha}\}$ with finite multiplicity:
	\begin{equation}\label{X15}
		\sum_{k=1}^\infty \chi_{B_k}(x)\leq K,
	\end{equation}
	Writing \eqref{X14} for the system of balls $B_k$, it follows 
	\begin{equation}\label{X16}
		\begin{aligned}
			v(B_k\cap \Omega_{2\alpha})&\leq \frac{C\gamma^\sigma }{1-C\gamma^\sigma }v(B_k\cap \Omega_\alpha)+ \frac{2^{3q-1}A_{pq}^q}{\gamma^{2q}\alpha^q}\sum_{j=1}^{2}\left(\int_{B_k\cap(\Om_\alpha-\Om_{2\alpha})}{\big|\partial_{z_j}u\big|}^p\omega_j(z)dz\right)^{\frac{q}{p}}.
		\end{aligned}
	\end{equation}
	Summing \eqref{X16} over $k$ and taking into account \eqref{X15}, and using elementary inequality for positive numbers $a_1^{q/p}+a_2^{q/p}+\dots\leq (a_1+a_2+\cdots)^{q/p}$, we arrive at 
	\begin{equation}\label{X17}
		v(\Omega_{2\alpha})\leq \frac{KC\gamma^\sigma }{1-C\gamma^\sigma }v(\Omega_\alpha)+ \frac{K2^{3q-1}A_{pq}^q}{\gamma^{2q}\alpha^q}\sum_{j=1}^{2}\left(\int_{ \Om_\alpha-\Om_{2\alpha}}{\big|\partial_{z_j}u\big|}^p\omega_j(z)dz\right)^{\frac{q}{p}}.
	\end{equation}
	We integrate \eqref{X17} over $(0,\infty)$:
	\begin{equation}\label{X18}
		\begin{aligned}
			\int_0^\infty v(\Omega_{2\alpha})d\alpha^q &\leq \frac{KC\gamma^\sigma }{1-C\gamma^\sigma }\int_0^\infty v(\Omega_\alpha)d\alpha^q+ \sum_{j=1}^{2}\frac{K2^{3q-1}qA_{pq}^q}{\gamma^{2q}}\int_0^\infty \frac{d\alpha}{\alpha}\left(\int_{ \Omega_\alpha -\Omega_{2\alpha}}{\big|\partial_{z_j}u\big|}^p\omega_j(z)dz\right)^{\frac{q}{p}}.
		\end{aligned}
	\end{equation}
	Since 
	\begin{equation*}
		\begin{split}
			\int_0^\infty v(\Omega_{2\alpha})d\alpha^q=\frac{1}{2^q}\int_{\Omega^i}u^q(x)v(x)dx,\quad
			\int_0^\infty v(\Omega_{\alpha})d\alpha^q=\int_{\Omega^i}u^q(x)v(x)dx. 
		\end{split}
	\end{equation*}
	We can apply the Minkowski's inequality in \eqref{X18} to obtain that 
	
	\begin{equation}\label{X19}
		\begin{aligned}
			\bigg(&\frac{1}{2^q}-\frac{KC\gamma^\sigma }{1-C\gamma^\sigma }\bigg)\int_{\Omega^i}u^q(x)v(x)dx\leq \frac{K2^{3q-1}qA_{pq}^q}{\gamma^{2q}}\sum_{j=1}^{2}\left(\int_{\Omega^i }{\big|\partial_{z_j}u\big|}^p\omega_j(z)\left(\int_{u(z)/2}^{u(z)}\frac{d\alpha}{\alpha}\right)^{\frac{p}{q}}dz\right)^{\frac{q}{p}}.
		\end{aligned}
	\end{equation}
	Choosing $\gamma$ so small that
	\begin{equation}\label{X20}
		\frac{1}{2^q}-\frac{KC\gamma^\sigma }{1-C\gamma^\sigma }>0, 
	\end{equation}
	Therefore, \eqref{X19} implies that 
	\begin{equation}\label{X21}
		\begin{split}
			\int_{\Omega^i}u^q(x)v(x)dx
			\leq &\left(\frac{1}{2^q}-\frac{KC\gamma^\sigma }{1-C\gamma^\sigma }\right)^{-1}\frac{K2^{3q-1}q{\rm ln2}}{\gamma^{2q}}A_{pq}^q\sum_{j=1}^{2}\left(\int_{\Omega^i }{\big|\partial_{z_j}u\big|}^p\omega_j(z)dz\right)^{\frac{q}{p}}.
		\end{split}
	\end{equation}
	Summing the inequalities \eqref{X21} for all $\Omega^i$ and using elementary inequality for positive numbers $a_1^{q/p}+a_2^{q/p}+\dots\leq (a_1+a_2+\cdots)^{q/p}$, it follows
	\begin{equation}\label{X22}
		\begin{split}
			\int_{\Omega^+}u^q(x)v(x)dx
			&\leq C_0^qA_{pq}^q\sum_{j=1}^{2}\left(\int_{\Omega^+ }{\big|\partial_{z_j}u\big|}^p\omega_j(z)dz\right)^{\frac{q}{p}},
		\end{split}
	\end{equation}
	with
	\begin{equation*}
		C_0^q= \left(\frac{1}{2^q}-\frac{KC\gamma^\sigma }{1-C\gamma^\sigma }\right)^{-1}\frac{K2^{3q-1}q}{\gamma^{2q}}{\rm ln2}.
	\end{equation*} 
	Similarly, we can prove the inequality is true in $\Omega^-$ for the function $-u(x)$:
	\begin{equation}\label{X23}
		\int_{\Omega^-}(-u(x))^qv(x)dx
		\leq C_0^qA_{pq}^q\sum_{j=1}^{2}\left(\int_{\Omega^- }{\big|\partial_{z_j}u\big|}^p\omega_j(z)dz\right)^{\frac{q}{p}}.\\
	\end{equation} 
	
	Combining estimates \eqref{X22} and \eqref{X23}, we obtain the inequality:
	\begin{equation}
		{\left(\int_{\Omega}{\left\vert u(x)\right\vert}^qv(x)dx\right)}^{\frac{1}{q}}\leq C_0A_{pq}\sum_{j=1}^2\left(\int_{\Omega}{\left\vert \partial_{x_j} u\right\vert}^p\omega_j(x)dx\right)^{\frac{1}{p}},  
	\end{equation}
 so far, the proof is completed.			
\end{proof}

\begin{lemma}\label{Le.3.2}
	For any $u\in \mcH_{h,0}^1(\Omega)$ and  $2\le q \le \frac{2}{1-\theta}$,   we have 
	\begin{equation}\label{Se.3.7}
	\|u\|_{L^q(\Omega)}\le C\|u\|_{\mcH_{h,0}^1(\Omega)}.	
	\end{equation}
\end{lemma}

\begin{proof}
We see that $\mcH_{h,0}^1(\Omega)=\overline{{\rm Lip}_0(\Omega)}^{\mcH_h^1(\Omega)}$. Indeed, we shall only to prove that $\overline{{\rm Lip}_0(\Omega)}^{\mcH_h^1(\Omega)} \subset \overline{C_0^\infty(\Omega)}^{\mcH_h^1(\Omega)}$ since $\overline{C_0^\infty(\Omega)}^{\mcH_h^1(\Omega)} \subset \overline{{\rm Lip}_0(\Omega)}^{\mcH_h^1(\Omega)} $ is obvious. For any $u\in\overline{{\rm Lip}_0(\Omega)}^{\mcH_h^1(\Omega)}$,  there exists a sequence $\{u_n\}_{n\in\mathbb{N}}\subset{\rm Lip}_0(\Omega)$ such that $\|u_n-u\|_{\mcH_h^1(\Omega)}\rightarrow 0$ as $n \rightarrow \infty$ and $|Du_n|\leq C \mbox{ a.e. on } \Omega$, which shows that  $u_n\in H_0^1(\Omega)$ with compact support on $\Omega$, then there exists $\{v_n\}_{n\in\mathbb{N}}\subset C_0^\infty(\Omega)$ such that 
$\|v_n-u_n\|_{H_0^1(\Omega)}\rightarrow 0, \ n\rightarrow \infty$,  
and we have 
\begin{equation*}
	\|v_n-u_n\|_{\mcH_h^1(\Omega)}\leq C\|v_n-u_n\|_{H^1(\Omega)}\rightarrow 0, \ n\rightarrow \infty
\end{equation*}
in view of the definition of $h(x_1)$, in addition
\begin{equation*}
	\|v_n-u\|_{\mcH_h^1(\Omega)}\leq \|v_n-u_n\|_{\mcH_h^1(\Omega)}+\|u_n-u\|_{\mcH_h^1(\Omega)}
	\rightarrow 0, \ n\rightarrow \infty,
\end{equation*}
so that $\mcH_{h,0}^1(\Omega)=\overline{{\rm Lip}_0(\Omega)}^{\mcH_h^1(\Omega)}$, ${\rm Lip}_0(\Omega)$ is dense in space $\mcH_{h,0}^1(\Omega)$. 

It is not hard to verify that $v(x)=1$ meets the conditions  of $ A_\infty(\Omega)$.  Consider $w_1=1$ and $w_2=h(x_1)$, and $p=2$ in \eqref{Le.3.1.e1}, therefore for $p=2$, $2\le q \le \frac{2}{1-\theta}$, we immediately receive that  
\begin{equation*}
\begin{split}
|B|^{-1} d(B)  v(B\cap\Omega)^{\frac{1}{q}} &\left({\int_{B\cap \Omega} w_1(x)^{-1}dx}\right)^{\frac{1}{2}} \leq Cr^\frac{2}{q}	\le C,\\
|B|^{-1} d(B)  v(B\cap\Omega)^{\frac{1}{q}} &\left({\int_{B\cap \Omega} w_2(x)^{-1}dx}\right)^{\frac{1}{2}} \le Cr^{\theta-1+\frac{2}{q}} \le C,
\end{split}	
\end{equation*}
 furthermore, by Lemma \ref{Le.3.1}  for any $u \in \mcH_{h,0}^1(\Omega)$
\begin{equation*}
\left(\int_{\Omega} |u|^qdx\right)^{\frac{1}{q}}\le C \left[\left(\int_{\Omega}|\partial_{x_1}u|^2dx\right)^{\frac{1}{2}}+\left(\int_{\Omega}h(x_1)|\partial_{x_2}u|^2 dx\right)^{\frac{1}{2}}\right]. 	
\end{equation*}
By Jensen inequality $\frac{a_1^t+a_2^t}{2}\le  \left(\frac{a_1+a_2}{2} \right)^t$ for $t\in (0,1)$,  $a_1,a_2 \in \mathbb{R^+}$, we have 
\begin{equation*}
\left(\int_{\Omega}|\partial_{x_1}u|^2dx\right)^{\frac{1}{2}}+\left(\int_{\Omega}  h(x_1)|\partial_{x_2}u|^2 dx\right)^{\frac{1}{2}}	\le C\left(\int_{\Omega}|\partial_{x_1}u|^2+h(x_1)|\partial_{x_2}u|^2dx\right)^{\frac{1}{2}},
\end{equation*}
so that the desired result is proved.
\end{proof}

\begin{remark}\label{202306071004}
	In particular, Theorem {\rm \ref{Le.3.2}} implies {\rm Poincar\'e} inequality, i.e.
	\begin{equation}\label{202302092134}
		\|u\|_{L^2(\Omega)}\le C\left(\int_{\Omega}|\partial_{x_1}u|^2+h(x_1)|\partial_{x_2}u|^2 dx\right)^{\frac{1}{2}}, \  \forall u\in \mcH_{h,0}^1(\Omega).	
	\end{equation}
	Therefore,  the equivalent norm of $\mcH_{h,0}^1(\Omega)$ could be written as:
	\begin{equation*}
		\|u\|_{\mcH_{h,0}^1(\Omega)}=\left(\int_{\Omega}|\partial_{x_1}u|^2+h(x_1)|\partial_{x_2}u|^2 dx\right)^{\frac{1}{2}}.
	\end{equation*}	
\end{remark}

\begin{theorem}\label{Th.3.3}
The  embedding $\mcH_{h,0}^1(\Omega)\hookrightarrow L^2(\Omega)$ is compact.	
\end{theorem}

\begin{proof}
Firstly, let $\{u_n\}_{n\in\mb{N}}\subset\mcH_{h,0}^1(\Omega)$ be a bounded sequence, we  assert that the  embedding $\mcH_{h,0}^1(\Omega)\hookrightarrow L^1(\Omega)$ is compact, so  $\{u_n\}_{n\in\mb{N}}$ is Cauchy in $L^1(\Omega)$. Set 
\begin{equation*}
	\tilde u_n(x)=
	\begin{cases}
		u_n(x), & x\in \Omega,\\
		0, & x\in \mb{R}^2\backslash\Omega,
	\end{cases}
\end{equation*}
then $\tilde u_n\in\mcH_{h,0}^1(\mb{R}^2)$ according to  Lemma \ref{L.2.4}. In order to show $\{\tilde u_n\}_{n\in\mb{N}}$ is paracompact in $L^1(\Omega)$, we shall show that {\cite[Theorem 2.32]{bookAdams}}: for every  $\epsilon>0$, there exists $\delta>0$ and a subset $G\subset\subset \Omega$, such that 
\begin{equation}\label{Se.3.8}
	\int_\Omega|\tilde u_n(x+\xi)-\tilde u_n(x)|dx<\epsilon,\quad \int_{\Omega-\ol G}|\tilde u_n(x)|dx<\epsilon
\end{equation}
for each  $n\in\mb{N}$ and $\xi=(\xi_1,\xi_2)\in\mb{R}^2$, $|\xi|<\delta$. Let $\epsilon\in (0,1)$ be arbitrary, for each  $n\in\mb{N}$, there exists $\{\hat u_n\}_{n\in\mb{N}}\subset C_0^\infty(\Omega)$, such that 
\begin{equation}\label{Se.3.9}
	\|u_n-\hat u_n\|_{\mcH_h^1(\Omega)}<\epsilon, 
\end{equation}
still denoted by $\widetilde{(\hat {u_n})}=\hat {u_n}$ for each $\hat {u_n}\in C_0^\infty(\Omega)$. Obviously, $\{\hat u_n\}_{n\in\mb{N}}\subset C_0^\infty(\Omega)$ is bounded in $\mcH_{h,0}^1(\mb{R}^2)$, note that
\begin{equation*}
	\begin{split}
		\hat u_n(x_1,{x_2})
		&=\int_0^{x_2}(\pt_{x_2}\hat u_n)(x_1,t)\df t=\int_0^{x_2} \sqrt{h^{-1}(x_1)}\left(\sqrt{h(x_1)}\pt_{x_2}\hat u_n\right)(x_1,t) dt,
	\end{split}
\end{equation*}
which imply
\begin{equation}\label{Se.3.10}
	\begin{split}
		\int_0^1|\hat u_n(s,{x_2})|\df s
		&\leq\int_0^1 \left|\int_0^{x_2} \sqrt{h^{-1}(s)} \sqrt{h(s)}\pt_{x_2}\hat u_n(s,t)\df t\right| ds\\
		&\leq \|h^{-1}\|_{L^1(0,1)}^\frac{1}{2}\|\sqrt{h(\cdot)}\pt_{x_2}\hat u_n\|_{L^2(\Omega)}x_2^\frac{1}{2}, 
	\end{split}
\end{equation} 
we deduce that 
\begin{equation*}
	\begin{split}
		\int_{(0,1)\times (0,j^{-1})}|\hat u_n(x)|dx
		&\leq \|h^{-1}\|_{L^1(0,1)}^\frac{1}{2}\|\sqrt{h(\cdot)}\pt_{x_2}\hat u_n\|_{L^2(\Omega)}\int_0^{j^{-1}} x_2^\frac{1}{2} d{x_2}\\
		&=\frac{2}{3}\|h^{-1}\|_{L^1(0,1)}^\frac{1}{2}\|\sqrt{h(\cdot)}\pt_{x_2}\hat u_n\|_{L^2(\Omega)}j^{-\frac{3}{2}}
	\end{split}
\end{equation*}
with $j\in\mb{N}$.  The same argument is developed again, we have 
\begin{equation*}
	\begin{split}
		\int_{(0,1)\times(1-j^{-1}, 1)}|\hat u_n(x)|dx
		&\leq \frac{2}{3}\|h^{-1}\|_{L^1(0,1)}^\frac{1}{2}\|\sqrt{h(\cdot)}\pt_{x_2}\hat u_n\|_{L^2(\Omega)} \left(1-(1-\frac{1}{j})^{\frac{3}{2}}\right),\\
		\int_{(0,j^{-1})\times (0,1)}|\hat u_n(x)|dx
		&\leq\frac{2}{3}\|\pt_{x_1}\hat u_n\|_{L^2(\Omega)}j^{-\frac{3}{2}},\\
		\int_{(1-j^{-1},1)\times (0,1)}|\hat u_n(x)|dx
		&\leq\frac{2}{3}\|\pt_{x_1}\hat u_n\|_{L^2(\Omega)}\left(1-(1-\frac{1}{j})^{\frac{3}{2}}\right). 
	\end{split}
\end{equation*}
Set 
\begin{equation*}
	\Omega_j=\{(x_1,{x_2})\in \Omega\mid x_1,{x_2}\in (0,j^{-1})\cup(1-j^{-1}, 1)\}, \ j\in\mb{N},
\end{equation*}
then 
\begin{equation*}
	\int_{\Omega_j}|\hat u_n(x)|dx\leq C\|\hat u_n\|_{\mcH_{h,0}^1(\Omega)}\max\left\{j^{-\frac{3}{2}},\left(1-(1-\frac{1}{j})^{\frac{3}{2}}\right)\right\},
\end{equation*}
where $C>0$ is a constant dependent only on $h$, which implies that there exists $j_0\in\mb{N}$, when $j\geq j_0$, we have 
\begin{equation*}
	\int_{\Omega_j}|\hat u_n(x)|dx<\epsilon, 
\end{equation*}
taking $G=\Omega-\ol \Omega_{j_0}$, we have proved that 
\begin{equation}\label{Se.3.11}
\int_{\Omega-\ol G}|\hat u_n(x)|dx<\epsilon.	
\end{equation}
 Let $\xi=(\xi_1,\xi_2)\in\mb{R}^2$, without loss of generality, we assume $\xi_1>0, \xi_2>0$, since
\begin{equation*}
	\begin{split}
		|\hat u_n(x_1+\xi_1,{x_2}+\xi_2)-\hat u_n(x_1,{x_2}+\xi_2)|
		&=\left|\int_{x_1}^{x_1+\xi_1}\pt_{x_1}\hat u_n(s,{x_2}+\xi_2) ds\right|\\
		&\leq \xi_1^\frac{1}{2}\left(\int_{x_1}^{x_1+\xi_1}\left|\pt_{x_1}\hat u_n(s,{x_2}+\xi_2)\right|^2 ds\right)^\frac{1}{2}\\
		&\leq \xi_1^\frac{1}{2}\left(\int_{\mb{R}}\left|\pt_{x_1}\hat u_n(s,{x_2}+\xi_2)\right|^2 ds\right)^\frac{1}{2},
	\end{split}
\end{equation*}
we obtain that 
\begin{equation*}
	\int_{\mb{R}}|\hat u_n(x_1+\xi_1,{x_2}+\xi_2)-\hat u_n(x_1,{x_2}+\xi_2)|d{x_2}\leq \xi_1^\frac{1}{2}\|\pt_{x_1}\hat u_n\|_{L^2(\Omega)}.  
\end{equation*}
hence
\begin{equation}\label{Se.3.12}
	\int_{\mb{R}^2}|\hat u_n(x_1+\xi_1,{x_2}+\xi_2)-\hat u_n(x_1,{x_2}+\xi_2)|dx\leq \xi_1^\frac{1}{2}\|\pt_{x_1}\hat u_n\|_{L^2(\Omega)}.  
\end{equation}
By the same way as \eqref{Se.3.10}, we obtain that 
\begin{equation*}
	\int_{\mb{R}}|\hat u_n(x_1,{x_2}+\xi_2)-\hat u_n(x_1,{x_2})|d{x_2}\leq \xi_2^\frac{1}{2}\|h^{-1}\|_{L^1(0,1)}^\frac{1}{2}\|\sqrt{h(\cdot)}\pt_{x_2}\hat u_n\|_{L^2(\Omega)}.  
\end{equation*}
therefore
\begin{equation}\label{Se.3.13}
	\int_{\mb{R}^2}|\hat u_n(x_1,{x_2}+\xi_2)-\hat u_n(x_1,{x_2})|dx\leq \xi_2^\frac{1}{2}\|h^{-1}\|_{L^1(0,1)}^\frac{1}{2}\|\sqrt{h(\cdot)}\pt_{x_2}\hat u_n\|_{L^2(\Omega)}.   
\end{equation}
From
\begin{equation*}
	\begin{split}
		|\hat u_n(x+\xi)-\hat u_n(x)|
		&\leq |\hat u_n(x_1+\xi_1,{x_2}+\xi_2)-\hat u_n(x_1, {x_2}+\xi_2)|+|\hat u_n(x_1,{x_2}+\xi_2)-\hat u_n(x_1,{x_2})|
	\end{split}
\end{equation*}
and \eqref{Se.3.12}, \eqref{Se.3.13} to lead 
\begin{equation*}
	\int_{\mb{R}^2}|\hat u_n(x+\xi)-\hat u_n(x)|dx \leq \xi_1^\frac{1}{2}\|\pt_{x_1}\hat u_n\|_{L^2(\Omega)}+\xi_2^\frac{1}{2}\|h^{-1}\|_{L^1(0,1)}^\frac{1}{2}\|\sqrt{h(\cdot)}\pt_{x_2}\hat u_n\|_{L^2(\Omega)},
\end{equation*}
this implies that
\begin{equation}\label{Se.3.14}
\int_{\mb{R}^2}|\hat u_n(x+\xi)-\hat u_n(x)|dx<\epsilon	
\end{equation}
for $|\xi|=\sqrt{\xi_1^2+\xi_2^2}$ small enough. 	Combining inequalities \eqref{Se.3.9}, \eqref{Se.3.11} and \eqref{Se.3.14}, we can get
 \begin{equation*}
 	\begin{split}
 		&\int_{\mb{R}^2}|\tilde u_n(x_1+\xi_1,{x_2}+\xi_2)-\tilde u_n(x_1,{x_2})|dx \\
 		&=\int_{\mb{R}^2}|\tilde u_n(x_1+\xi_1,{x_2}+\xi_2)-\hat u_n(x_1+\xi_1,{x_2}+\xi_2)|dx+\int_{\mb{R}^2}|\hat u_n(x_1+\xi_1,{x_2}+\xi_2)-\hat u_n(x_1,{x_2})|dx\\
 		&\hspace{6mm}+\int_{\mb{R}^2}|\hat u_n(x_1,{x_2})-\tilde u_n(x_1,{x_2})|dx\\
 		&\leq 2\|\tilde u_n-\hat u_n\|_{L^2(\mb{R}^2)}+\int_{\mb{R}^2}|\hat u_n(x_1+\xi_1,{x_2}+\xi_2)-\hat u_n(x_1,{x_2})|dx\\
 		&\leq 3\epsilon,
 	\end{split}
 \end{equation*}
and 
\begin{equation*}
		\begin{split}
			\int_{\Omega-\ol G}|\tilde u_n(x_1, {x_2})|dx
			&\leq \int_{\Omega-\ol G}|\tilde u_n(x_1,{x_2})-\hat u_n(x_1,{x_2})|dx+\int_{\Omega-\ol G} |\hat u_n(x_1,{x_2})|dx\leq 2\epsilon.
		\end{split}
	\end{equation*} 
so far the inequality \eqref{Se.3.8} is verified, then we know that $\{u_n\}_{n\in\mb{N}}$ is Cauchy in $L^1(\Omega)$.

According to Lemma \ref{Le.3.2}, for $  q=\frac{2}{1-\theta}>2$, we have $\{u_n\}_{n\in\mb{N}}$ is bounded in $L^q(\Omega)$. By interpolation inequality:
\begin{equation}\label{Se.3.15}
	\|u_n-u_m\|_{L^2(\Omega)}\le \|u_n-u_m\|_{L^1(\Omega)}^\alpha \|u_n-u_m\|_{L^q(\Omega)}^{1-\alpha}\le (2C)^{1-\alpha}\|u_n-u_m\|_{L^1(\Omega)}^{\alpha}
\end{equation} 
for some $\alpha \in (0,1)$, given that $\{u_n\}_{n\in\mb{N}}$ is Cauchy in $L^1(\Omega)$, choosing $n,m$ sufficiently large shows that $\{u_n\}_{n\in\mb{N}}$ is Cauchy in $L^2(\Omega)$ by \eqref{Se.3.15}, this argument completes the proof of Theorem \ref{Th.3.3}. 
\end{proof}

\section{Weak solution and its regularity}

\begin{definition}\label{D.2.6}
	(i) The {\it bilinear form} $B[\cdot,\cdot]$ associated with the sub-elliptic operator $L$ is 
	\begin{equation*}
		B[u,v]=\int_\Omega (\pt_{x_1}u)(\pt_{x_1}v)+\left(\sqrt{h(x_1)}\pt_{x_2}u\right)\left(\sqrt{h(x_1)}   \pt_{x_2}v  \right) + Vuv dx
	\end{equation*}
	for $u,v\in \mcH_{h,0}^1(\Omega)$.
	
	(ii) We call $u\in\mcH_{h,0}^1(\Omega)$ is a {\it weak solution} to the sub-elliptic equation \eqref{e1}, if 
	\begin{equation*}
		B[u,v]=(f, v)_{L^2(\Omega)}
	\end{equation*}
	for all $v\in \mcH_{h,0}^1(\Omega)$. 
\end{definition}

\begin{theorem}\label{T.2.6}
	Let $f\in L^2(\Omega)$  be arbitrary given function, then the sub-elliptic equation \eqref{e1} has a unique solution $u\in \mcH_{h,0}^1(\Omega)$. 
\end{theorem}

\begin{proof}
	From Lemma \ref{L.2.1} we know that $(\mcH_{h,0}^1(\Omega),(\cdot,\cdot)_{\mcH_{h}^1(\Omega)})$ is a Hilbert space. Given that Remark \ref{202306071004}  we have
	\begin{equation*}
		B[u,u]=\int_\Omega |\pt_{x_1}u|^2+ \left|\sqrt{h(x_1)}\pt_{x_2} u\right|^2+ Vu^2dx
		\ge \frac{1}{2}\|u\|_{\mcH_{h,0}^1(\Omega)}^2		
	\end{equation*}
	for all $u\in\mcH_{h,0}^1(\Omega)$. Moreover, thanks to the H\"older inequality, we deduce that 
	\begin{align*}
		\begin{split}
			&|B[u,v]|\\
			&\leq \|\pt_{x_1} u\|_{L^2(\Omega)}\|\pt_{x_1}v\|_{L^2(\Omega)}+\|\sqrt{h(\cdot)}\pt_{x_2} u\|_{L^2(\Omega)}\|\sqrt{h(\cdot)} \pt_{x_2}v     \|_{L^2(\Omega)}+M\|u\|_{L^2(\Omega)}\|v\|_{L^2(\Omega)}\\
			&\leq \max\{1,M\} \left( \|\pt_{x_1} u\|_{L^2(\Omega)}\|\pt_{x_1}v\|_{L^2(\Omega)}+\|\sqrt{h(\cdot)}\pt_{x_2} u\|_{L^2(\Omega)}\|\sqrt{h(\cdot)}    \pt_{x_2}v   \|_{L^2(\Omega)}+\|u\|_{L^2(\Omega)}\|v\|_{L^2(\Omega)}\right)\\
			&\leq C\left(\|u\|_{L^2(\Omega)}^2+\|\pt_{x_1}u\|_{L^2(\Omega)}^2+\|\sqrt{h(\cdot)}\pt_{x_2}u\|_{L^2(\Omega)}^2\right)^\frac{1}{2}\left(\|v\|_{L^2(\Omega)}^2+\|\pt_{x_1}v\|_{L^2(\Omega)}^2+\|\sqrt{h(\cdot)}\pt_{x_2}v\|_{L^2(\Omega)}^2\right)^\frac{1}{2}\\
			&=C \|u\|_{\mcH_{h,0}^1(\Omega)}\|v\|_{\mcH_{h,0}^1(\Omega)}
		\end{split}
	\end{align*}
	for all $u,v\in \mcH_{h,0}^1(\Omega)$.  For each  $v\in \mcH_{h,0}^1(\Omega)$, it is briefly find that $f: \mcH_{h,0}^1(\Omega)\ra \mb{R}$ is a linear functional, and note that
	\begin{equation*}
		(f, v)_{L^2(\Omega)} \leq \|f\|_{L^2(\Omega)}\|v\|_{L^2(\Omega)}\leq \|f\|_{L^2(\Omega)}\|v\|_{\mcH_{h,0}^1(\Omega)},
	\end{equation*}
	therefore $f$ is a bounded linear functional on $\mcH_{h,0}^1(\Omega)$. Applying the Lax-Milgram theorem \cite{EvansPartial}, the desired result is proved.
\end{proof}

We denote $\Om'\subset\subset\Om$ by ${\Om'}\subset\Om$ is an open subset of $\Om$ and $\overline{\Om'}$ is a compact subset of $\Om$. 

\begin{theorem}\label{Th.k.1}
	$(1)$  Let $D_i^lu(x)=\frac{u(x+le_i)-u(x)}{l} \ (i=1,2)$ denote the $i\raisebox{0mm}{-}th$ difference quotient of size $l$ for $x\in \Omega'$, $l\in \mathbb{R}$, $0<|l|<\dist(\Omega',\partial{\Omega})$. Suppose $u\in \mcH_{h,0}^1(\Omega)$, then for any $\Omega' \subset\subset \Omega$, we have 
	\begin{equation*}
		\|D^lu\|_{L^2(\Omega')}	\le C\|\sqrt{h(x_1)}Du\|_{L^2(\Omega)}\le C\|u\|_{\mcH_{h,0}^1(\Omega)}
	\end{equation*}	
	for some constant $C$ and all $0<|l|<\frac{1}{2}\dist(\Omega',\partial{\Omega})$, where $D^lu=(D_1^lu,D_2^lu)$.
	
	$(2)$ Let $D^lu$ be defined as shown in $(1)$.  Suppose  $u\in L^2(\Omega')$, and there is a constant $C$ such that
	\begin{equation*}
		\|D^lu\|_{L^2(\Omega')}\le C
	\end{equation*}
	for all $0<|l|<\frac{1}{2}\dist(\Omega',\partial{\Omega})$, then $u\in H^1(\Omega')$ with $\|Du\|_{L^2(\Omega')}\le C.$
\end{theorem}

\begin{proof}
	$(1)$ Suppose $u\in C_0^\infty(\Omega)$, for any $x\in \Omega'$, $i=1,2$, $0<|l|<\frac{1}{2}\dist(\Omega',\partial{\Omega})$, we see that
	\begin{equation*}
		u(x+le_i)-u(x)=\int_0^1 \partial_{x_i}u(x+tle_i)dt\cdot le_i,
	\end{equation*}
	so that 
	\begin{equation*}
		|u(x+le_i)-u(x)|\le |l|\int_0^1 |\partial_{x_i}u(x+tle_i)|dt, 
	\end{equation*}
	i.e.
	\begin{equation*}
		|D_i^lu(x)|\le \int_0^1 |\partial_{x_i}u(x+tle_i)|dt.	
	\end{equation*}
	Therefore for any $\delta \in (0,1)$, by Cauchy inequality, we obtain that 
	\begin{equation*}
		\begin{split}
			\inf_{x_1\in [\delta,1-\delta]}
			&h(x_1)  
			\int_{\Omega'} |D^lu(x)|^2dx\\
			&\le \int_{\Omega'} h(x_1)|D^lu(x)|^2dx =\int_{\Om'}h(x_1)\left(|D^l_1u(x)|^2+|D^l_2u(x)|^2\right)dx\\
			& \leq \int_{\Om'}h(x_1)\sum_{i=1}^2\left(\int_0^1|\partial_{x_i}u(x+tle_i)|dt \right)^2dx\le \sum_{i=1}^{2}\int_{\Omega'}h(x_1)\int_0^1|\partial_{x_i}u(x+tle_i)|^2dtdx\\
			&= \int_0^1\int_{\Om'}h(x_1) |\partial_{x_1}u(x_1+tl,x_2)|^2+h(x_1)|\partial_{x_2}u(x_1, x_2+tl)|^2 dxdt\\
			&\le C \left(\int_\Om |\partial_{x_1}u(x)|^2dx+\int_\Om |\sqrt{h(x_1)}\partial_{x_2}u(x)|^2dx\right)=C\|u\|_{\mcH_{h,0}^1(\Omega)}^2. 
		\end{split}
	\end{equation*}
	Since $C_0^\infty(\Omega)$ is dense in $\mcH_{h,0}^1(\Omega)$, the above inequality is established for any $u\in \mcH_{h,0}^1(\Omega)$.

	$(2)$ For this proof, we can refer to Theorem 3 \cite[Chapter 5.8.2]{EvansPartial}.	
\end{proof}

\begin{theorem}\label{Th.k.2}
	Assuming that $h(x_1)$ and $V(x)$ are defined as given in the Introduction (Section 1), and $f\in L^2(\Omega)$. Suppose that $u\in \mcH_{h,0}^1(\Omega)$ is the weak solution for the problem $Lu=f$ on $\Omega$, then $u\in H_{\rm loc}^2(\Omega)$. 
\end{theorem}

\begin{proof}
	1. For any subset $\Omega'\subset\subset \Omega$, we may choose an open set $W$ such that $\Omega' \subset \subset W \subset \subset \Omega$. In addition, we select a  function $ \eta \in C_0^\infty(\mathbb{R}^2)$	such that $\eta \equiv1$ on $\Omega'$, $\eta \equiv 0$ on $\mathbb{R}^2- W$ and $0\le\eta \le 1$.
	Note that $u$ is the weak solution for $Lu=f$, then for any $v\in \mcH_h^1(\Omega)$:
	\begin{equation}\label{Th.2.7.1}
		\sum_{i,j=1}^{2}\int_{\Omega} a_{ij}(x)\pt_{x_i}u \pt_{x_j}vdx=\int_{\Omega} fv-Vuvdx,
	\end{equation}
	where $a_{11}=1,a_{12}=a_{21}=0,a_{22}=h(x_1)$. Let 
	\begin{equation*}
		A=\sum\limits_{i,j=1}^{2}\int_{\Omega} a_{ij}(x)\pt_{x_i}u \pt_{x_j}vdx,\quad B=\int_{\Omega} fv-Vuvdx, 
	\end{equation*}
	then $A=B$.
	
	2. Let $0<|l|<\frac{1}{2} \min\{\dist(\Omega',\pt W),\dist(W,\pt \Omega)\} $ and consider that $|l|$ be sufficiently small, then substitute $v=-D_k^{-l}(\eta^2D_k^lu)$ into \eqref{Th.2.7.1}, $k\in \{1,2\}$. Then
		\begin{align*}
			A=&	-\sum\limits_{i,j=1}^{2}\int_{\Omega} a_{ij}(x)\pt_{x_i}u \pt_{x_j}\left(D_k^{-l}(\eta^2D_k^lu)\right)dx\\
			&=\sum\limits_{i,j=1}^{2}\int_{\Omega} D_k^l(a_{ij}\pt_{x_i}u)	\pt_{x_j}(\eta^2D_k^lu)dx\\
			&=\sum\limits_{i,j=1}^{2}\int_{\Omega}  a_{ij}^l(D_k^l\pt_{x_i}u)\pt_{x_j}(\eta^2D_k^lu)+(D_k^la_{ij})\pt_{x_i}u\pt_{x_j}(\eta^2D_k^lu)	dx\\
			&=\sum\limits_{i,j=1}^{2}\int_{\Omega}a_{ij}^l(D_k^l\pt_{x_i}u)(D_k^l\pt_{x_j}u) \eta^2dx+\sum\limits_{i,j=1}^{2}\int_{\Omega}\bigg(2\eta (\pt_{x_j}\eta) a_{ij}^l(D_k^l\pt_{x_i}u)D_k^lu\\
			&  \quad\quad \quad +(D_k^la_{ij})(\pt_{x_i}u)(D_k^l\pt_{x_j}u)\eta^2+2\eta (\pt_{x_j}\eta)(D_k^la_{ij})(\pt_{x_i}u)D_k^lu\bigg)dx\\
			&=A_1+A_2,
		\end{align*}
	where $a_{ij}^l(x)=a_{ij}(x+le_k)$. According to the definitions of $h(x_1)$ and $\eta$, we have
	\begin{equation}\label{Th.2.7.3}
		A_1\ge \min\left\{1,\inf_{x_1\in [\delta,1-\delta]}{h(x_1)}\right\} \int_{\Omega} \eta^2 |D_k^lDu|^2dx \ge \theta  \int_{\Omega} \eta^2 |D_k^lDu|^2dx
	\end{equation}
	for some proper constant $\theta$ and $\delta \in (0,1)$, and 
	\begin{equation*}
		|A_2|\le C \int_{\Omega} \left(\eta |D_k^lDu||D_k^lu|+\eta|D_k^lDu||Du|+\eta|D_k^lu||Du|\right)dx. 
	\end{equation*}
	Furthermore, by  $\supp \eta\subset \overline{W}$ and Cauchy's inequality with $\epsilon$, then
	\begin{equation*}
		|A_2|\le \epsilon \int_{\Omega} \eta^2|D_k^lDu|^2dx+\frac{C}{\epsilon}\int_{W} \left(|D_k^lu|^2+|Du|^2\right)dx.
	\end{equation*}
	By invoking the result $(1)$ of Theorem \ref{Th.k.1}, we have $\int_{W} |D_k^lu|^2 \le C\|u\|_{\mcH_{h,0}^1(\Omega)}^2$. It is not hard to find that 
	\begin{equation*}
	\int_{W} |Du|^2dx\le
	\big(\min\{1,\inf\limits_{x_1\in [\delta,1-\delta] }h(x_1)\}\big)^{-1}\|u\|_{\mcH_{h,0}^1(\Omega)}^2  \le C\|u\|_{\mcH_{h,0}^1(\Omega)}^2 
	\end{equation*}
	for $\delta \in (0,1)$. We may choose $\epsilon=\frac{\theta}{2}$, hence
	\begin{equation}\label{Th.2.7.4}
		|A_2|\le \frac{\theta}{2} \int_{\Omega} \eta^2 |D_k^lDu|^2dx + C \|u\|_{\mcH_{h,0}^1(\Omega)}^2.
	\end{equation}
	Combining \eqref{Th.2.7.3} with \eqref{Th.2.7.4}, we can get
	\begin{equation}\label{Th.2.7.5} 
		A \ge \frac{\theta}{2} \int_{\Omega} \eta^2 |D_k^lDu|^2dx -C\|u||_{\mcH_{h,0}^1(\Omega)}^2.	
	\end{equation}
	
	3. Since $v=-D_k^{-l}(\eta^2D_k^lu)$, by the $(1)$ of Theorem \ref{Th.k.1}, 
	\begin{equation*}
		\begin{split}
			\int_{\Omega} |v|^2dx &\le  C \int_{\Omega} \left|\sqrt{h(x_1)}D\left(\eta^2 D_k^lu\right)\right|^2dx\\
			 &\le C\int_{W} |D_k^lu|^2+\eta^2|D_k^lDu|^2 dx\\
			&\le  C\|u||_{\mcH_{h,0}^1(\Omega)}^2+C \int_{\Omega} \eta^2|D_k^lDu|^2dx.
		\end{split}
	\end{equation*}
	By the Cauchy's inequality with $\epsilon$, we have
	\begin{equation*}
		\begin{split}		
			|B|&\le C\int_{\Omega} |f||v|+|u||v| dx\\
			&\le \epsilon \int_{\Omega} \eta^2|D_k^lDu|^2dx+\frac{C}{\epsilon}\int_{\Omega}|f|^2+|u|^2 dx+\frac{C}{\epsilon}\|u||_{\mcH_{h,0}^1(\Omega)}^2.
		\end{split}	
	\end{equation*}
	Similarly, we choose $\epsilon=\frac{\theta}{4}$, then
	\begin{equation}\label{Th.2.7.6}
		|B|\le 	\frac{\theta}{4} \int_{\Omega} \eta^2|D_k^lDu|^2dx+C\int_{\Omega} |f|^2+|u|^2 dx+C\|u||_{\mcH_{h,0}^1(\Omega)}^2.
	\end{equation}
	
	4. According to \eqref{Th.2.7.5} and \eqref{Th.2.7.6}, we obtain that
	\begin{equation*}
		\frac{\theta}{4} \int_{\Omega} \eta^2|D_k^lDu|^2dx\le C\int_{\Omega} |f|^2+|u|^2 dx+C\|u||_{\mcH_{h,0}^1(\Omega)}^2,
	\end{equation*}
	therefore 
	\begin{equation*}
	 \int_{\Omega'} |D_k^lDu|^2dx\le C\int_{\Omega} |f|^2+|u|^2 dx+C\|u||_{\mcH_{h,0}^1(\Omega)}^2, \ k=1,2,
	\end{equation*}
for any sufficiently small $|l| \neq 0$. Using the result $(2)$ of Theorem \ref{Th.k.1}, we know $Du\in H^1(\Omega')$, hence $u\in H_{\rm loc}^2(\Omega)$.
\end{proof}

\section{Fundamental gap}

This section is mainly to characterize the optimal potentials over the class $S$, some of ideas were developed in \cite{Ashbaugh1991ON,hislop2012introduction,pankov2006introduction,bookPerturbation,karaa1998extremal}. Before that,  we will obtain basic properties of spectral theory for the boundary-value problem:
\begin{equation}\label{Se3.1}
	\begin{cases}
		Lu=\lambda u, & x\in \Omega, \\
		u=0, & x\in\partial \Omega,
	\end{cases} 
\end{equation}
where the operator $L$ is defined as \eqref{e3}. 

\begin{lemma}\label{Le.3.4}
	For any given $f\in L^2(\Omega)$, define
	\begin{equation}
		S:L^2(\Omega)\rightarrow L^2(\Omega),\ f\mapsto u, 
	\end{equation}
	where $u \in \mcH_{h,0}^1(\Omega)$ is the unique solution for differential equation  \eqref{e1}, then  $S$ is a bounded self-adjoint compact operator.		
\end{lemma}

\begin{proof}
	For any $f,g\in L^2(\Omega)$, let $u$ and $v$ be the solutions corresponding to $f$ and $g$, respectively, we observe that
	\begin{equation*}
		(Sf,g)_{L^2(\Omega)}=(u,g)_{L^2(\Omega)}=B[u,v]=(f,v)_{L^2(\Omega)}=(f,Sg)_{L^2(\Omega)}.
	\end{equation*}
	And 
	\begin{equation*}
		B[u,u]=(f,u)_{L^2(\Omega)}\leq \|f\|_{L^2(\Omega)}\|u\|_{L^2(\Omega)}\leq \|f\|_{L^2(\Omega)}\|u\|_{\mcH_{h,0}^1(\Omega)}, 
	\end{equation*}
	this implies that  $\|u\|_{\mcH_{h,0}^1(\Omega)}\leq C \|f\|_{L^2(\Omega)}$, therefore
	\begin{equation*}
		\|Sf\|_{L^2(\Omega)}=\|u\|_{L^2(\Omega)}\leq \|u\|_{\mcH_{h,0}^1(\Omega)}\leq C\|f\|_{L^2(\Omega)}. 
	\end{equation*}	
	To make a connection with the result obtained in Theorem \ref{Th.3.3}, the Lemma \ref{Le.3.4} is proved.	
\end{proof}

\begin{remark}\label{R.3.5}
	It is not hard to find that $\dim \mcH_{h,0}^1(\Omega)=\infty$, so that $0\in \sigma(S), \sigma(S)-\{0\}=\sigma_p(S)-\{0\}$,  and  $\sigma(S)-\{0\}$ is a sequence tending to $0$, where $\sigma_p(S)$ is recorded as the point spectrum of $S$ and $\sigma(S)$ is considered as the spectrum of $S$. 
\end{remark}

\begin{lemma}\label{Le.3.6}
$(1)$ All the eigenvalues of $L$ is real and can be arranged in a monotone sequence on the basis of its (finite) multiplicity: 
	\begin{equation*}
	\sigma(L)=\{\lambda_k\}_{k=1}^{\infty}, \quad  0<\lambda_1 \le \lambda_2\le \lambda_3\le \cdots \le \lambda_k \cdots \rightarrow \infty, \quad k\rightarrow \infty.
	\end{equation*}

$(2)$ There exists an  orthonormal basis $\{w_k\}_{k=1}^{\infty}\subset L^2(\Omega)$, where $w_k\in\mcH_{h,0}^1(\Omega)$ is an eigenfunction with respect to $\lambda_k$
\begin{equation*}
		\begin{cases}
			Lw_k=\lambda_kw_k, & x\in \Omega, \\
		    u=0, & x\in\partial \Omega,
	\end{cases} 
\end{equation*}
for $k=1,2 \cdots. $

$(3)$ We have
\begin{equation}\label{Le.3.6.e3}
	\lambda_k=\inf_{\substack{E \subset\mcH_{h,0}^1(\Omega)\\  d(E)=k }} \sup_{\substack{u \in E \\ \|u\|_{L^2{(\Omega)}} = 1}} 	B[u,u].	
\end{equation}
In particular, assuming that we have already computed $u_1,u_2,\cdots, u_{k-1}$ the $(k-1)\raisebox{0mm}{-}th$ first eigenfunctions, we also have:
 $\lambda_k=\inf\{B[u,u]\mid u\in \mcH_{h,0}^1(\Omega), u\bot V_{k-1}, \|u\|_{L^2(\Omega)}=1\}$, where $V_{k-1}=span\{u_1,u_2\cdots,u_{k-1}\}$, the equality holds if and only if $u=w_k$.

$(4)$ The eigenvalue $\lambda_1$ is simple and the first eigenfunction  $u_1$ is positive on $\Omega$.
\end{lemma}

\begin{proof}
	Let $S=L^{-1}$, $S$ is a self-adjoint compact linear operator owing to Lemma \ref{Le.3.4}.  Note
\begin{equation*}
	(Lf,f)=(u,f)=B[u,u]\ge 0 
\end{equation*}
for any given $f \in L^2(0,1)$.	 According to Theorem D7 (pp. 728) \cite{EvansPartial}, we know that the eigenvalues of $S$ is real and positive, and there exists a countable orthonormal basis of $L^2(\Omega)$ consisting of eigenvectors of $S$. Moreover, for $\eta \neq 0$, $Sw=\eta w$ if and only if $Lw=\lambda w$ for $\lambda=\frac{1}{\eta}$. Then $(1)(2)$ is proved.

By $(2)$, we have 
\begin{equation}\label{Le.3.6.1}
	\begin{cases}
		B[w_k,w_k]=\lambda_k(w_k,w_k)=\lambda_k,\\
		B[w_k,w_l]=\lambda_k(w_k,w_l)=0,k,l=1,2\cdots ,k\neq l.
	\end{cases}	
\end{equation}	
Since $\{w_k\}_{k=1}^\infty$ is the orthogonal basis of $L^2(\Omega)$, if $u\in \mcH_{h,0}^1(\Omega)$ and $\|u\|_{L^2(\Omega)}=1$, then
\begin{equation}\label{4.5.2}
	u=\sum_{k=1}^{\infty}d_kw_k, \quad  d_k=(u,w_k), \quad \sum_{k=1}^{\infty}d_k^2=\|u\|_{L^2(\Omega)}^2=1,
\end{equation}
the series converging in $L^2(\Omega)$. By \eqref{Le.3.6.1}, $\frac{w_k}{\sqrt{\lambda_k}}$ is an orthonormal subset of $\mcH_{h,0}^1(\Omega)$, endowed with the new inner product $B[\cdot,\cdot]$. Actually, due to $V(x)\in S$,
\begin{equation}\label{4.5.3}
	c_1\|u\|_{\mcH_{h,0}^1(\Omega)}^2\le B[u,u]\le  c_2\|u\|_{\mcH_{h,0}^1(\Omega)}^2.
\end{equation}
Besides, if $B[w_k,u]=\lambda_k(w_k,u)=0$ for $k=1,2\cdots $, we deduce that $u \equiv 0$ due to $\lambda_k>0$ and $\{w_k\}_{k=1}^{\infty}$ is the orthonormal basis of $L^2(\Omega)$. Hence $u=\sum\limits_{k=1}^{\infty}\mu_k \frac{w_k}{\sqrt{\lambda_k}}$ for $\mu_k=B[u,\frac{w_k}{\sqrt{\lambda_k}}]$, the series converging in $\mcH_{h,0}^1(\Omega)$. Combining with \eqref{4.5.2}, we obtain that $\mu_k={d_k}{\sqrt{\lambda_k}}$ and $u=\sum\limits_{k=1}^{\infty}d_kw_k$ is convergent in $\mcH_{h,0}^1(\Omega)$.  

Let $E$ denotes any $k\raisebox{0mm}{-}$dimensional subspace of $\mcH_{h,0}^1(\Omega)$, and let $U={\rm span}\{w_k,w_{k+1}, \cdots\}$, clearly, $E \cap U \neq \emptyset$. For any $u \in E \cap U$ with $\|u\|_{L^2(\Omega)}=1$, we have $u=\sum\limits_{i=k}^{\infty}d_iw_i$ with $\sum\limits_{i=k}^{\infty}d_i^2=1$, furthermore, 
\begin{equation*}
	B[u,u]=\sum\limits_{i=k}^{\infty}d_i^2B(w_i,w_i)=\sum\limits_{i=k}^{\infty}d_i^2\lambda_i \ge \lambda_k,
\end{equation*}	
in other words, $\sup\limits_{\substack{u \in E \\ \|u\|_{L^2{(\Omega)}} = 1}} B[u,u] \ge \lambda_k$ is established for any subspace $E \subset \mcH_{h,0}^1(\Omega)$, i.e.
\begin{equation}\label{Le.3.6.2}
	\inf_{\substack{E \subset\mcH_{h,0}^1(\Omega)\\  d(E)=k }} \sup_{\substack{u \in E \\ \|u\|_{L^2{(\Omega)}} = 1}} B[u,u] \ge \lambda_k.
\end{equation}	
On the other side, consider  $V={\rm span}\{w_1,w_2,\cdots,w_k\}$, we find that $V\subset E$. For any $u \in V$ with $\|u\|_{L^2(\Omega)}=1$, we obtain 
\begin{equation*}
	B[u,u]=\sum\limits_{i=1}^{k}d_i^2B[w_i,w_i]=\sum\limits_{i=1}^{k}d_i^2\lambda_i \le \lambda_k,
\end{equation*}	
that is, $\sup\limits_{\substack{u \in V \\ \|u\|_{L^2{(\Omega)}} = 1}} B[u,u] \le \lambda_k$, given that $V$ is one of the subspaces of $E$, then
\begin{equation}\label{Le.3.6.3}
\inf_{\substack{E \subset\mcH_{h,0}^1(\Omega)\\  d(E)=k }} \sup_{\substack{u \in E \\ \|u\|_{L^2{(\Omega)}} = 1}} 	B[u,u] \le \lambda_k.
\end{equation}	
By employing the equalities \eqref{Le.3.6.2} and \eqref{Le.3.6.3}, then \eqref{Le.3.6.e3} is proved.

Suppose we have got $u_1,u_2,\cdots, u_{k-1}$, for any $u\in \mcH_{h,0}^1(\Omega) $ and $u\bot V_{k-1}$  with $\|u\|_{L^2(\Omega)}=1$, where $V_{k-1}={\rm span}\{u_1,u_2\cdots,u_{k-1}\}$, then $u=\sum\limits_{i=1}^{\infty}(u,w_i)w_i=\sum\limits_{i=k}^{\infty}(u,w_i)w_i=\sum\limits_{i=k}^{\infty}d_iw_i$ and $\sum\limits_{i=k}^{\infty}d_i^2=1$. Moreover, $B[u,u]=\sum\limits_{i=k}^{\infty} B[w_i,w_i]= \sum\limits_{i=k}^{\infty} d_i^2\lambda_i \ge \lambda_k $, which implies that 
 $\inf\limits_{\substack{u \in \mcH_{h,0}^1(\Omega),u\bot V_{k-1} \\ \|u\|_{L^2{(\Omega)}} = 1}} B[u,u] \ge \lambda_k$. By \eqref{Le.3.6.1}, we see that $B[w_k,w_k]=\lambda_k$, then the result $(3)$ is confirmed.

 For the weak solution  $u \in \mcH_{h,0}^1(\Omega)$, actually $u \in W^{1,2}(\Omega')$  and the elliptic operator $L$ are uniform for any $\Omega' \subset\subset \Omega$ by the  definition of $h(x_1)$, according to Chapter 8.6 and Chapter 8.8 on \cite{gilbarg2015elliptic}, we can get the Harnack inequality   $\sup\limits_{x\in \Omega'}u(x)\leq C\inf\limits_{x\in \Omega'}u(x)$. 	According to $(3)$, if $u$ is the eigenfunction for $\lambda_1$, so is $|u|$. Since $|u|$ is non-negative in $\Omega$, we further obtain that $|u|$ is a positive eigenfunction taking advantage of Harnack inequality. This shows that the eigenfunctions of $\lambda_1$ are either positive or negative and thereby it is impossible that two of them are orthogonal, therefore $\lambda_1$ is simple. 
\end{proof}

 Next, we commence by investigating  the existence of the minimum fundamental gap when $V(x)$ is limited to the set $S$, and then further characterize the optimal function and express its behavior.
 
\begin{theorem} \label{T.1}
The fundamental gap $\Gamma (V)$ attains its minimum in the classes of $S$ by $V^*$.
\end{theorem}

\begin{proof}
Let $\{V^k\}_{k\in \mathbb{N} } \in S$ be the minimization sequence of $\Gamma(V)$, i.e.
\begin{equation*}
	\Gamma(V^k)\downarrow \inf\limits_{V\in S } \Gamma(V). 
\end{equation*}	
Since  $S$ is a bounded closed convex set, there is a sequence $\{V^k\}_{k\in\mathbb{N}}$ and $V^* \in S$,  such that 
	\begin{equation}\label{Se.4.1}
		V^k\rightarrow V^* \mbox{ weakly star in } L^\infty(\Omega). 
	\end{equation}	
	
Let $\{(\lambda_j^k,u_j^k)\}_{k\in\mathbb{N}}$ be a sequence of eigenpairs of fixed index of the degenerate elliptic equation \eqref{Se3.1} related to $V^k$, where $u_j^k$ has unit $L^2$ norm, $j=1,2$. We may assume $h(x_1)\le H$ for $x_1\in (0,1)$ a.e. so that
\begin{equation}
	\begin{split}
\lambda_j^k=&\inf_{\substack{E \subset\mcH_{h,0}^1(\Omega)\\  d(E)=j }} \sup_{\substack{u \in E \\ \|u\|_{L^2{(\Omega)}} = 1}} {\int_\Omega |\partial_{x_1}u|^2+|\sqrt{h(x_1)}\partial_{x_2}u|^2+V^k|u|^2dx}\\
&\le \inf_{\substack{E \subset\mcH_{h,0}^1(\Omega)\\  d(E)=j }} \sup_{\substack{u \in E \\ \|u\|_{L^2{(\Omega)}} = 1}}{\int_\Omega |\partial_{x_1}u|^2+H|\partial_{x_2}u|^2+M|u|^2dx}\\
&={\int_\Omega  |\partial_{x_1}\widetilde u_j|^2+H|\partial_{x_2}\widetilde u_j|^2+M|\widetilde u_j|^2dx}\le C,
	\end{split}		
\end{equation}
where $\|\widetilde u_j\|_{L^2(\Omega)}$ is the corresponding  normalized eigenfunction of \eqref{Se3.1} involving with $h(x_1) \equiv H$ and $V(x) \equiv M$ on $\Omega$, $j=1,2$.  In fact, 
	\begin{equation*}
	\frac{1}{2}\|u_j^k\|_{\mcH_{h,0}^1(\Omega)}^2	\le{\int_\Omega |\partial_{x_1}u_j^k|^2+|\sqrt{h(x_1)}\partial_{x_2}u_j^k|^2+V|u_j^k|^2 dx} =\lambda_j^k \le C, 
	\end{equation*}
	therefore
	\begin{equation}\label{Se.4.2}
		\|u_j^k\|_{\mcH_{h,0}^1(\Omega)} \le C. 
	\end{equation}
 So, there exists a subsequence of $\{u_j^k\}_{k\in\mathbb{N}}$ by \eqref{Se.4.2},  still preserved this index, such that
	\begin{equation}\label{Se.4.3}
		u_j^k\rightarrow u_j^* \mbox{ weakly in } \mcH_{h,0}^1(\Omega), \ j=1,2,
	\end{equation}
	and by Theorem \ref{Th.3.3}
	\begin{equation}\label{Se.4.4}
		u_j^k\rightarrow u_j^* \mbox{ strongly in } L^2(\Omega), \ j=1,2.
	\end{equation}
	Now we can extract a further subsequence such that 
	\begin{equation}\label{Se.4.5}
		\lambda_j^k\rightarrow \lambda_j^*, \ j=1,2,
	\end{equation}
for each $j$ we know for all  $v \in \mcH_{h,0}^1(\Omega)$                  
	\begin{equation}\label{Se.4.6}
	{\int_\Omega(\partial_{x_1}u_j^k)(\partial_{x_1}v) +(\sqrt{h(x_1)}\partial_{x_2}u_j^k)(\sqrt{h(x_1)}\partial_{x_2}v)+V^ku_j^k vdx} =\lambda_j^k \int_\Omega u_j^k vdx, 
	\end{equation}
choosing $k$ sufficiently large, according to \eqref{Se.4.1} \eqref{Se.4.3}  and \eqref{Se.4.4} \eqref{Se.4.5},
               
		\begin{equation}\label{Se.4.7}
			{\int_\Omega(\partial_{x_1}u_j^*)(\partial_{x_1}v) +(\sqrt{h(x_1)}\partial_{x_2}u_j^*)(\sqrt{h(x_1)}\partial_{x_2}v)+V^*u_j^* vdx} =\lambda_j^* \int_\Omega u_j^* vdx. 
		\end{equation}

 This indicates that $\lambda_j^k$ converges to an element of the spectrum of the problem \eqref{Se3.1} given by $V^*$. In particular, we may extract a subsequence of  $\{u_1^k\}_{k\in \mathbb{N}}\subset L^2(\Omega)$ such that $u_1^k$ is converge to $u_1^*$ a.e. by \eqref{Se.4.4}, based on the non-negativity of $u_1^k$ (see Lemma \ref{Le.3.6} (4)), then $u_1^*$ must be the first eigenfunction and $\lambda_1^*=\lambda_1(V^*)$. Considering that $u_2^*$ and $u_1^*$ are orthogonal on $\Omega$, we know that $u_2^*$ must change the sign on $\Omega$, which means that  $u_2^*$ not to be the first eigenfunction of $\lambda_1(V^*)$ by Lemma \ref{Le.3.6}. Therefore, we have $\lambda_2^*\ge \lambda_{2}(V^*)$, furthermore,
 \begin{equation}
 	\Gamma(V^k) \rightarrow \lambda_2^*-\lambda_1^*\ge \Gamma(V^*), k \rightarrow \infty,
 \end{equation}
that is, $\Gamma(V^*) \le \inf\limits_{V\in S } \Gamma(V)$. Meanwhile,  $\Gamma(V^*) \ge \inf\limits_{V\in S } \Gamma(V)$, then  we conclude that  the minimum value  of $\Gamma(V)$ can be reached.
\end{proof}

\begin{definition}\label{202306091636}
	A real-valued, measurable and bounded  function $P(x)$ on $\Omega$ is called an {\it admissible perturbation} of $V(x)$ provided $V(x)+tP(x) \in S$ for any sufficiently small $t\in (-\epsilon,\epsilon)$. The function $P(x)$ is called a {\it left-admissible} ({\it right-admissible}) perturbation provided $V(x)+tP(x) \in S$ for non-positive (non-negative) $t$.
\end{definition}

Consider the  operator $L$ as in \eqref{e3}  on the Hilbert space $\mcH_{h,0}^1(\Omega)$ and a family of operators $L_t$ defined by replacing $V$ by $V+tP(x)$ in L, $t\in (-\epsilon,\epsilon)$ for small $\epsilon>0$. Let $\lambda_0$ be a discrete eigenvalue of $L_0$, then there exist families $\lambda_l(t)\ (l=1,2,\cdots,r)$ of discrete eigenvalues of $L_t$ such that $\lambda_l(0)=\lambda_0$, the total multiplicity of the eigenvalues $\lambda_l(t)\ (l=1,2,\cdots,r)$ is equal to the multiplicity of $\lambda_0$, and the eigenvalues are analytic in $t$. Especially, if the discrete eigenvalue $\lambda_j(L_0)$ is simple \cite{hislop2012introduction,pankov2006introduction,bookPerturbation} , we have 
\begin{equation}\label{Le.4.2}
	\frac{d\lambda_{j}(V+tP(x))}{dt} \Big|_{t=0}=\int_\Omega P(x) u_j^2 dx, 
\end{equation}
where $u_j$ is the normalized eigenfunction with respect to $L_0$. If the discrete eigenvalue $\lambda_j(L_0)$ is degenerate (the multiplicity is $r$) \cite{hislop2012introduction,pankov2006introduction,bookPerturbation,RichardCourant,rellich1969perturbation} , then it splits into several eigenvalue branches $\lambda_{j,m}$ under a perturbation, these are the eigenvalues of $L_t$ which converge to $\lambda_j$ as $t \rightarrow 0$, and each branch is an analytic function for small $t$, but those functions do not ordinarily correspond to the ordering of eigenvalues given by the min-max principle (Lemma \ref{Le.3.6} (3)). For example, the lowest one for $t<0$ will be the highest for $t>0$. Moreover, we 
\begin{equation}\label{Le.4.2022}
	\frac{d\lambda_{j,m}(V+tP(x))}{dt} \Big|_{t=0}=\int_\Omega P(x) u_{j,m}^2 dx, 
\end{equation} 
where the orthonormal eigenfunctions $u_{j,m}$  are chosen so that
\begin{equation*}
	\int_\Omega u_{j,i} P(x) u_{j,m} =0, \ i\neq m.	
\end{equation*}

\begin{theorem}\label{Th.4.4}
If a function $V^*$	minimizes $\Gamma(V)$ in the class $S$ for the eigenvalue problem \eqref{Se3.1}, then $\lambda_2(V^*)$ is non-degenerate  and there exists a subset $\omega \subset \Omega$ such that 
\begin{equation*}
V^*=m\chi_\omega+M\chi_{\omega^c},	
\end{equation*}
moreover, 
\begin{alignat*}{2}
	|u_2^*|^2-|u_1^*|^2 
	&\ge 0 \   && \mbox{\rm on} \  \omega,\\
	|u_2^*|^2-|u_1^*|^2
	&\le 0 \  \  &&  \mbox{\rm  on} \ \omega^c ,				
\end{alignat*}
where $u_1^*$ and $u_2^*$ represent the first and second normalized eigenfunctions with respect to $V^*$, respectively.  
\end{theorem}

\begin{proof}
 {\it Step 1}.    Suppose $\lambda_2(V^*)$ is simple, we claim that $T:=\{x\in \Omega\mid m<  V^*(x) <M\}$ is a subset with zero measure. 
 
By contradictory, we suppose $|T|>0$. Let $T_k=\{x\in \Omega\mid m+\frac{1}{k}<V^*(x)<M-\frac{1}{k}\}$ for all $k\in \mathbb{N}$. Since $T=\bigcup\limits_{k = 1}^\infty T_{k}$ and $T_k\subset T_{k+1}$ for all $k\in\mb{N}$, there exists $k_0\in\mb{N}$ such that  $|T_{k_0}|>0$. Taking  $0<\epsilon<\frac{1}{k_0}$, for any Lebesgue point $x_0 \in T_{k_0}$ and any measurable sequence of subsets $R_{k,j} \subset T_{k_0}$ containing $x_0$,
we see that $P(x)=\chi_{R_{k_0,j}}$ is  an admissible perturbation on sufficiently small $t\in (-\epsilon,\epsilon)$.  By  \eqref{Le.4.2}, 
  \begin{equation}\label{Se.4.9}
  	\frac{d\Gamma(V+tP(x))}{dt}\Big|_{t=0}=\int_{R_{k_0,j}} |u_2^*|^2-|u_1^*|^2  dx=0,
  \end{equation}
dividing the $|R_{k_0,j}|$ and letting $R_{k_0,j}$ shrink nice to $x_0$ as $j \to \infty$,  from the Lebesgue Differential Theorem we have  $|u_2^*|^2-|u_1^*|^2=0$ on $T_{k_0}$ and hence on $T$.

{\it Step 2}. Next, we show that $|u_2^*|^2=|u_1^*|^2$ can only appear on the set with zero measure. Otherwise, based on $(4)$ of Lemma \ref{Le.3.6}, we have $T=T^{+} \cup T^{-}$, $T^+=\{x\in T\mid u_2^*(x)>0\}, T^-=\{x\in T\mid u_2^*(x)<0\}$,  without loss of generality, we assume the set $T^+$ is a set of positive measure. Then we have $u_1^*-u_2^*=0$ on $T^+$, note that $u_1^*-u_2^* \in H_{loc}^2(\Omega)$ by Theorem \ref{Th.k.2}, then $\Delta(u_1^*-u_2^*)= 0$ a.e. on $T^+$ \cite[(Corollary 1.21, Chapter 1)]{NonlinearPotential}. Furthermore, we can get $Lu_1^*=Lu_2^*$ a.e. on $T^+$. Together with equation \eqref{Se3.1}, we obtain that $(\lambda_2-\lambda_1)u_1^*=0$ a.e. on $T^+$, in view of result $(4)$ of Lemma \ref{Le.3.6}, then $\lambda_2=\lambda_1$, it is a contradictory. 

{\it Step 3}. If $\lambda_2(V^*)$ is simple, for any Lebesgue point $x_0 \in \omega:= \{x\in \Omega\mid V^*(x)=m\}$ and $G_j$  is a sequence of sets containing $x_0$, we see that $P(x)=\chi_{G_j}$ is  a right admissible perturbation  for sufficiently small non-negative $t$, and  repeating the proof of Step 1 then we obtain $|u_2^*|^2-|u_1^*|^2\ge 0$ on $\omega$. Similarly, we can receive the properties of $\omega^c$ employing the same technique for $V(x)=M$. According to the  normalization condition $ \|u_1^*\|_{L^2(\Om)}^2=\|u_2^*\|_{L^2(\Om)}^2=1$,
 we see that
    \begin{equation}\label{202306091626}
  \int_\Omega |u_2^*|^2-|u_1^*|^2dx=0
  \end{equation} 
 we observe that $\omega$ and $\omega^c$ are non-empty. Indeed, if the set $\omega^c$ is empty, we must have $|u_2^*|^2-|u_1^*|^2\ge 0$ on $\Omega$, this is in contradiction with \eqref{202306091626}. Similarly, if $\omega$ is empty, it will also cause contradiction. 
  
{\it Step 4}. Suppose $\lambda_2(V^*)$ is $r$-fold degenerate, we first realize that the set $T$ mentioned in Step 1 is a subset of zero measure, and we will give a proof by contradiction. In fact, $\lambda_2(V^*)$ will split into several eigenvalue branches $\lambda_{2,m}$ under an admissible perturbation $P(x)$, each branch is an analytic function for small $t$
at $t=0$, and by \eqref{Le.4.2022} 
\begin{equation}\label{Se.4.11}
\frac{d\Gamma_{m}(V+tP(x))}{dt} \Big|_{t=0}=\int_\Omega P(x)\left(|u_{2,m}^*|^2-|u_1^*|^2\right) dx,	
\end{equation}
where $\Gamma_{m}=\lambda_{2,m}-\lambda_1$ and the orthonormal eigenfunctions $u_{2,m}^*$  are specially chosen so that $\int_\Omega u_{2,j}^* P(x)u_{2,m}^*dx =0$ for $j \neq m$, $m=1,2,\cdots,r$.

 For any given proper admissible perturbation $P(x)$, if there exists $m$ such that   
 \begin{equation*}
  \frac{d\Gamma_{m}(V+tP(x))}{dt} \Big|_{t=0}>0 ,
\end{equation*}
 we would obtain that 
 \begin{equation}\label{202306091903}
 \Gamma(L_{t_0})\le \Gamma_m(t_0)< \Gamma_m(0)=\Gamma(0)	
 \end{equation}
for some negative $t_0$, however, $\Gamma(L_{0})$ is a minimum, so this situation is excluded. Analogously, if there exists $m$ such that   $\frac{d\Gamma_{m}(V+tP(x))}{dt} \Big|_{t=0}<0$, we would obtain that \eqref{202306091903}  for some positive $t_0$, this situation was also ruled out. Therefore,  the integral \eqref{Se.4.11} must vanish.

Suppose $u$ is any normalized eigenfunction in the eigenspace for $\lambda_2(V^*)$, then 
 \begin{equation*}
\begin{split}
&u^*=\sum\limits_{i=1}^{r} \alpha_i u_{2,i}^*,\\
&\sum\limits_{i=1}^{r} |\alpha_i|^2=1,
\end{split} 
\end{equation*}
where $\alpha_i \in \mathbb{R}$. According to the above conclusion, we can see that 
 \begin{equation}\label{Se.4.12}
 	\begin{split}
 	\int_\Omega P(x)\left(|u^*|^2-|u_1^*|^2\right) dx &=	\int_\Omega P(x) \left(\sum\limits_{i=1}^{r} |\alpha_i|^2 |u_{2,i}^*|^2-|u_1^*|^2\right)dx\\
 	&=	\int_\Omega P(x) \left(\sum\limits_{i=1}^{r} |\alpha_i|^2 |u_{2,i}^*|^2-\sum\limits_{i=1}^{r} |\alpha_i|^2|u_1^*|^2\right)dx\\
 	&=\sum\limits_{i=1}^{r} |\alpha_i|^2 \int_\Omega P(x) \left(|u_{2,i}^*|^2-|u_1^*|^2\right) dx =0.
 	\end{split} 
 \end{equation}
 As demonstrated in Step 1 and Step 2, we know the set $T$  is a subset of zero measure from \eqref{Se.4.12} by selecting the appropriate $P(x)$, that is,
 $V^*=m\chi_{\omega}+M\chi_{\omega^c}$. By the same way, we have 
\begin{equation}\label{Se.4.13}
|u^*|^2\ge |u_1^*|^2 \ \mbox{on} \ \omega,  \quad  |u^*|^2\le |u_1^*|^2 \ \mbox{on} \ \omega^c,	
\end{equation}
and $\omega$ and $\omega^c$ are non-empty.   

Take advantage of the Theorem \ref{Th.k.2}, we find that the weak solution of problem \eqref{Se3.1}  belongs to $H_{\rm loc}^2(\Omega)$, furthermore, the weak solution is continuous for any subset $\Omega'\subset \subset \Omega$ by Sobolev Embedding Theorem \cite{EvansPartial,gilbarg2015elliptic}. Then we may find a   interior point  $x_0$ on $\omega$ such that $u_{2,1}(x_0) \neq 0$ and $u_{2,2}(x_0) \neq 0$, where $u_{2,1}$ and $u_{2,2}$ are orthogonal on $\Omega$. Consider a nonzero vector $(\alpha_1, \alpha_2 )\in \mathbb{R}^2$, $\alpha_1^2+\alpha_2^2=1$, $\alpha_1 \neq 0,\alpha_2 \neq 0$  and $\alpha_1u_{21}(x_0)+\alpha_2u_{22}(x_0)=0$. Consider the normalized eigenfunction $u=\alpha_1u_{21}+\alpha_2u_{22}$, we have $u(x_0)=0$ inside $\omega$, this is incompatible with the fact of \eqref{Se.4.13}. Therefore, the second eigenvalue $\lambda_2(V^*)$ cannot be degenerate. 
\end{proof}

\section*{Acknowledgements}
This work was supported by the National Natural Science Foundation of China, the Science-Technology Foundation of Hunan Province.

\section*{Competing interests declaration}
We declare that we have no conflict of interest and we do not have any commercial or associative interest that represents a conflict of interest in connection with the work submitted.

\normalem

\bibliographystyle{plain}
\bibliography{ref.bib}

\begin{thebibliography}{10}

\bibitem{bookAdams}
R.~A. Adams and J.~J.~F. Fournier.
\newblock {\em Sobolev spaces}.
\newblock Pure and Applied Mathematics. Academic Press, 2 edition, 2003.

\bibitem{el2022optimal}
Z.~El Allali and E.~M. Harrell.
\newblock Optimal bounds on the fundamental spectral gap with single-well
  potentials.
\newblock {\em Proceedings of the American Mathematical Society},
  150(02):575--587, 2022.

\bibitem{Ben2011Proof}
B.~Andrews and J.~Clutterbuck.
\newblock Proof of the fundamental gap conjecture.
\newblock {\em Journal of the American Mathematical Society}, 24(3), 2011.

\bibitem{andrews2020non}
B.~Andrews, J.~Clutterbuck, and D.~Hauer.
\newblock Non-concavity of the robin ground state.
\newblock {\em Cambridge Journal of Mathematics}, 8(2):243--310, 2020.

\bibitem{andrews2021fundamental}
B.~Andrews, J.~Clutterbuck, and D.~Hauer.
\newblock The fundamental gap for a one-dimensional schr{\"o}dinger operator
  with robin boundary conditions.
\newblock {\em Proceedings of the American Mathematical Society},
  149(4):1481--1493, 2021.

\bibitem{ashbaugh1989optimal}
M.~S. Ashbaugh and R.~Benguria.
\newblock {Optimal lower bound for the gap between the first two eigenvalues of
  one-dimensional Schr{\"o}dinger operators with symmetric single-well
  potentials}.
\newblock {\em Proceedings of the American Mathematical Society},
  105(2):419--424, 1989.

\bibitem{Ashbaugh1991ON}
M.S. Ashbaugh, E.~M. Harrell, and R.~Svirsky.
\newblock On minimal and maximal eigenvalue gaps and their causes.
\newblock {\em Pacific Journal of Mathematics}, 147(1):1--24, 1991.

\bibitem{Bourni2021The}
T.~Bourni, J.~Clutterbuck, X.~H. Nguyen, A.~Stancu, and G.~F. Wei.
\newblock {The vanishing of the fundamental gap of convex domains in H-n}.
\newblock {\em Annales Henri Poincare}, 23(2):595--614, 2021.

\bibitem{cavalheiro2008weighted}
A.~C. Cavalheiro.
\newblock {Weighted Sobolev spaces and degenerate elliptic equations}.
\newblock {\em Boletim da Sociedade Paranaense de Matem{\'a}tica},
  26(1-2):117--132, 2008.

\bibitem{chen2014lower}
D.~Y. Chen and M.~J. Huang.
\newblock Lower bounds on the eigenvalue gap for vibrating strings.
\newblock {\em Journal of Mathematical Analysis and Applications},
  417(1):225--233, 2014.

\bibitem{ChenHua2019lower}
H.~Chen, H.~G. Chen, J.~F. Wang, and N.~N. Zhang.
\newblock {Lower bounds of Dirichlet eigenvalues for a class of higher order
  degenerate elliptic operators}.
\newblock {\em Journal of Pseudo-Differential Operators and Applications},
  10(2):475--488, 2019.

\bibitem{cheng2014dual}
Y.~H. Cheng, W.~C. Lian, and W.~C. Wang.
\newblock {The dual eigenvalue problems for p-Laplacian}.
\newblock {\em Acta Mathematica Hungarica}, 142(1):132--151, 2014.

\bibitem{2011ACompact}
S.~K. Chua, S.~Rodney, and R.~L. Wheeden.
\newblock A compact embedding theorem for generalized sobolev spaces.
\newblock {\em Pacific Journal of Mathematics}, 265(1):17--57, 2011.

\bibitem{RichardCourant}
R.~Courant and D.~Hilbert.
\newblock {\em Methods of Mathematical Physics}, volume~1.
\newblock Wiley-VCH, 1 edition, 1989.

\bibitem{dai2021fundamental}
X.~Z. Dai, S.~Seto, and G.~F. Wei.
\newblock Fundamental gap estimate for convex domains on sphere - the case n=2.
\newblock {\em Communications in Analysis and Geometry}, 29(5):1095--1125,
  2021.

\bibitem{van1983condensation}
M.~Van den Berg.
\newblock {On condensation in the free-boson gas and the spectrum of the
  Laplacian}.
\newblock {\em Journal of Statistical Physics}, 31(3):623--637, 1983.

\bibitem{duoandikoetxea2001fourier}
J.~Duoandikoetxea.
\newblock {\em Fourier analysis}, volume~29 of {\em Graduate Studies in
  Mathematics}.
\newblock American Mathematical Society, 2000.

\bibitem{EvansPartial}
L.~C. Evans.
\newblock {\em {Partial Differential Equations: Second Edition }}.
\newblock Graduate Studies in Mathematics. American Mathematical Society, 2
  edition, 2010.

\bibitem{evans2018measure}
L.~C. Evans and R.~F. Garzepy.
\newblock {\em Measure theory and fine properties of functions}.
\newblock Textbooks in mathematics (Boca Raton Fla.). CRC Press, revised
  edition, 2015.

\bibitem{gilbarg2015elliptic}
D.~Gilbarg and N.~S. Trudinger.
\newblock {\em Elliptic partial differential equations of second order}.
\newblock Springer, 2001.

\bibitem{he2020fundamental}
C.~X. He, G.~F. Wei, and Q.~S. Zhang.
\newblock Fundamental gap of convex domains in the spheres.
\newblock {\em American Journal of Mathematics}, 142(4):1161--1191, 2020.

\bibitem{Henrot2006Extremum}
A.~Henrot.
\newblock {\em Extremum problems for eigenvalues of elliptic operators}.
\newblock Frontiers in Mathematics. Birkhäuser Basel, 1 edition, 2006.

\bibitem{hislop2012introduction}
P.~D. Hislop and I.~M. Sigal.
\newblock {\em {Introduction to spectral theory: with applications to
  Schr{\"o}dinger operators}}, volume 113.
\newblock Springer Science \& Business Media, 2012.

\bibitem{horvath2003first}
M.~Horv{\'a}th.
\newblock {On the first two eigenvalues of Sturm-Liouville operators}.
\newblock {\em Proceedings of the American Mathematical Society},
  131(4):1215--1224, 2003.

\bibitem{2007The}
M.~J. Huang.
\newblock The eigenvalue gap for vibrating strings with symmetric densities.
\newblock {\em Acta Mathematica Hungarica}, 117(4):341--348, 2007.

\bibitem{NonlinearPotential}
T.~Kilpeläinen J.~Heinonen and O.~Martio.
\newblock {\em Nonlinear Potential Theory of Degenerate Elliptic Equations}.
\newblock Dover Books on Mathematics. Dover Publications, reprint edition,
  2006.

\bibitem{karaa1998extremal}
S.~Karaa.
\newblock {Extremal eigenvalue gaps for the Schr{\"o}dinger operator with
  Dirichlet boundary conditions}.
\newblock {\em Journal of Mathematical Physics}, 39(4):2325--2332, 1998.

\bibitem{bookPerturbation}
T.~Kato.
\newblock {\em Perturbation theory for linear operators}.
\newblock Classics in Mathematics. Springer, 1995.

\bibitem{keller1961lower}
J.B. Keller.
\newblock Lower bounds and isoperimetric inequalities for eigenvalues of the
  schr{\"o}dinger equation.
\newblock {\em Journal of Mathematical Physics}, (2):262--266, 1961.

\bibitem{kerner2021lower}
J.~Kerner.
\newblock {A lower bound on the spectral gap of one-dimensional Schr{\"o}dinger
  operators}.
\newblock {\em arXiv preprint arXiv:2102.03816}, 2021.

\bibitem{kerner2021lowerAlower}
J.~Kerner.
\newblock {A lower bound on the spectral gap of Schr{\"o}dinger operators with
  weak potentials of compact support}.
\newblock {\em arXiv preprint arXiv:2103.03813}, 2021.

\bibitem{lavine1994eigenvalue}
R.~Lavine.
\newblock The eigenvalue gap for one-dimensional convex potentials.
\newblock {\em Proceedings of the American mathematical Society},
  121(3):815--821, 1994.

\bibitem{mamedov2021poincare}
F.~Mamedov.
\newblock A poincare's inequality with non-uniformly degenerating gradient.
\newblock {\em Monatshefte f{\"u}r Mathematik}, 194(1):151--165, 2021.

\bibitem{mamedov2009some}
F.~Mamedov and R.~Amanov.
\newblock {On some nonuniform cases of the weighted Sobolev and Poincare
  inequalities}.
\newblock {\em St. Petersburg Mathematical Journal}, 20(3):447--463, 2009.

\bibitem{mamedov2014weighted}
F.~Mamedov and O.~Azizov.
\newblock {On Weighted Sobolev type inequalities in spaces of differentiable
  functions}.
\newblock {\em Azerbaijan Journal of Mathematics}, 4(2), 2014.

\bibitem{mamedov2018sawyer}
F.~Mamedov and Y.~Shukurov.
\newblock {A Sawyer-type sufficient condition for the weighted Poincar{\'e}
  inequality}.
\newblock {\em Positivity}, 22(3):687--699, 2018.

\bibitem{2009Maximum}
D.~D. Monticelli and K.~R. Payne.
\newblock Maximum principles for weak solutions of degenerate elliptic
  equations with a uniformly elliptic direction.
\newblock {\em Journal of Differential Equations}, 247(7):1993--2026, 2009.

\bibitem{Monticelli2020An}
D.~D. Monticelli and S.~Rodney.
\newblock {An improved compact embedding theorem for degenerate Sobolev
  spaces}.
\newblock {\em Matematiche (Catania)}, 75(1):259--275, 2020.

\bibitem{muckenhoupt1972weighted}
B.~Muckenhoupt.
\newblock {Weighted norm inequalities for the Hardy maximal function}.
\newblock {\em Transactions of the American Mathematical Society},
  165:207--226, 1972.

\bibitem{pankov2006introduction}
A.~Pankov.
\newblock {\em {Introduction to spectral theory of Schr{\"o}dinger operators}}.
\newblock Vinnitsa State Pedagogical University, 2006.

\bibitem{qi2020extremal}
J.~G. Qi, J.~Li, and B.~Xie.
\newblock Extremal problems of the density for vibrating string equations with
  applications to gap and ratio of eigenvalues.
\newblock {\em Qualitative Theory of Dynamical Systems}, 19(1):1--15, 2020.

\bibitem{rellich1969perturbation}
F.~Rellich and J.~Berkowitz.
\newblock {\em Perturbation theory of eigenvalue problems}.
\newblock CRC Press, 1969.

\bibitem{sawyer2010degenerate}
E.~Sawyer and R.~Wheeden.
\newblock Degenerate sobolev spaces and regularity of subelliptic equations.
\newblock {\em Transactions of the American Mathematical Society},
  362(4):1869--1906, 2010.

\bibitem{seto2019sharp}
S.~Seto, L.~L. Wang, and G.~F. Wei.
\newblock Sharp fundamental gap estimate on convex domains of sphere.
\newblock {\em Journal of Differential Geometry}, 112(2):347--389, 2019.

\bibitem{singer1985estimate}
I.~M. Singer, B.~Wong, S.~T. Yau, and S.~S-T Yau.
\newblock {An estimate of the gap of the first two eigenvalues in the
  Schr{\"o}dinger operator}.
\newblock {\em Annali della Scuola Normale Superiore di Pisa-Classe di
  Scienze}, 12(2):319--333, 1985.

\bibitem{Tan2022Estimates}
S.~Y. Tan and W.~J. Liu.
\newblock Estimates for eigenvalues of a class of fourth order degenerate
  elliptic operators with a singular potential.
\newblock {\em Journal of Mathematical Analysis and Applications},
  508(2):125907, 2022.

\bibitem{Wolfson2015eigenvalue}
J.~Wolfson.
\newblock Eigenvalue gap theorems for a class of nonsymmetric elliptic
  operators on convex domains.
\newblock {\em Communications in partial differential equations},
  40(4):601--628, 2015.

\bibitem{yau1986nonlinear}
S.~T. Yau.
\newblock Nonlinear analysis in geometry, monographies de lenseignement
  math{\'e}matique.
\newblock {\em L’Enseignement Math{\'e}matique, Geneva}, 33:54, 1986.

\end{thebibliography}

\end{document}